\documentclass{ws-m3as}
\graphicspath{{pix/}}
\begin{document}

\markboth{Alrachid and Leli\`evre}{Convergence of an ABF method: Variance reduction by Helmholtz projection}

%
\catchline{}{}{}{}{}
%

\title{Long-time convergence of an adaptive biasing force method:\\
  Variance reduction by Helmholtz projection}

\author{HOUSSAM ALRACHID}

\address{\'Ecole des Ponts ParisTech, Universit\'e Paris Est, 6-8 Avenue Blaise Pascal, Cit\'e Descartes \\
Marne-la-Vall\'ee,  F-77455 ,
France\\
alrachih@cermics.enpc.fr}
\address{Universit\'{e} Libanaise, Ecole Doctorale des Sciences et de Technologie, Campus Universitaire de Rafic Hariri\\
 Hadath, Lebanon\\
 houssam.alrachid@hotmail.com}

\author{TONY LELIEVRE}

\address{\'Ecole des Ponts ParisTech, Universit\'e Paris Est, 6-8 Avenue Blaise Pascal, Cit\'e Descartes \\
Marne-la-Vall\'ee,  F-77455 ,
France\\
lelievre@cermics.enpc.fr}
\maketitle

\begin{abstract}
In this paper, we propose an improvement of the adaptive biasing force (ABF) method, by projecting the estimated {\em mean force} onto a gradient. The associated stochastic process satisfies a non linear stochastic differential equation. Using entropy techniques, we prove exponential convergence to the stationary state of this stochastic process. We finally show on some numerical examples that the variance of the approximated {\em mean force} is reduced using this technique, which makes the algorithm more efficient than the standard ABF method.
\end{abstract}

\keywords{Adaptive biasing force; Helmholtz projection; Free energy; Variance reduction.}

\section{Introduction}	

\subsection{{\em The model}}
  Let us consider the {\em Boltzmann-Gibbs} measure :
\begin{equation}\label{gib}
\mu(dx)=Z_{\mu}^{-1}e^{-\beta V(x)}dx,
\end{equation}
where $x\in \mathcal{D}^N$ denotes the position of $N$ particles in $\mathcal{D}$. The space $\mathcal{D}$ is called the configuration space. One should think of $\mathcal{D}$ as a subset of $\mathbb{R}^n$, or the $n$-dimensional torus $\mathbb{T}^n$ (where $\mathbb{T}=\mathbb{R}/\mathbb{Z}$ denotes the one dimensional torus). The potential energy function $V:\mathcal{D} \longrightarrow \mathbb{R}$ associates to the positions of the particles $x\in \mathcal{D}$ its energy $V(x)$. In addition, $Z_{\mu}=\displaystyle \int_{\mathcal{D}} e^{-\beta V(x)}dx$ (assumed to be finite) is the normalization constant and $\beta=1/(k_{B}T)$ is proportional to the inverse of the temperature $T$, $k_{B}$ being the Boltzmann constant.

 The probability measure $\mu$ is the equilibrium measure sampled by the particles in the canonical statistical ensemble. A  typical dynamics that can be used to sample this measure is the {\em Overdamped Langevin Dynamics}:
\begin{equation}\label{over}
dX_t=-\nabla V(X_t)dt+\sqrt{\frac{2}{\beta}}dW_t,
\end{equation}
where $X_t\in \mathcal{D}^N$ and $W_t$ is a $Nn$-dimensional standard Brownian motion. Under loose assumptions on $V$, the dynamics $(X_t)_{t\geq 0}$ is ergodic with respect to the measure $\mu$, which means: for any smooth test function $\varphi$,
\begin{equation} \label{ergo}
\displaystyle\lim _{T\rightarrow +\infty}\frac{1}{T}\int_0^T\varphi(X_t)dt=\int\varphi d\mu,
\end{equation}
 i.e. trajectory averages converge to canonical averages.
\subsection{{\em Metastability, reaction coordinate and free energy}}\label{meta}

In many cases of interest, there exists regions of the configuration space where the dynamics \eqref{over} remains trapped for a long time, and jumps only occasionally to another region, where it again remains trapped for a long time. This typically occurs when there exist high probability regions separated by very low probability areas. The regions where the process $(X_t)_{t \ge 0} $ remains trapped for very long times, are called metastable. 

Because of the metastability, trajectorial averages \eqref{ergo} converge very slowly to their ergodic limit. Many methods have been proposed to overcome this difficulty, and we concentrate here on the Adaptive Biasing Force (denoted ABF) method (see [\refcite{dav:01,hen:04}]). In order to introduce the ABF method, we need another ingredient: a {\em reaction coordinate} (also known as an order parameter), $\xi=(\xi_1,...,\xi_m):\mathcal{D}\longrightarrow \mathbb{R}^m$, $\xi(x)=z$, where $m< nN$. Typically, in \eqref{over}, the time-scale for the dynamics on $\xi(X_t)$ is larger than the time-scale for the dynamics on $X_t$ due to the metastable states, so that $\xi$ can be understood as a function such that $\xi(X_t)$ is in some sense a slow variable compared to $X_t$. We can say that $\xi$ describes the metastable states of the dynamics associated to the potential $V$. 
For a given configuration $x$, $\xi(x)$ represents some macroscopic information. For example, it could represent angles or bond lengths in a protein, positions of defects in a material, etc ... 
 In any case, it is meant to be a function with values in a small dimensional space (i.e. $m\leq 4$), since otherwise, it is difficult to approximate accurately the associated free energy which is a scalar function defined on the range of $\xi$ (see equation \eqref{free0} below). The choice of a "good" reaction coordinate is a highly debatable subject in the literature. One aim of the mathematical analysis conducted here or in previous papers (see for example [\refcite{tony:08}]) is to quantify the efficiency of numerical algorithms once a reaction coordinate has been chosen.
 
The image of the measure $\mu$ by $\xi$ is defined by:
\begin{equation}\label{imxi}
 \xi*\mu:= \displaystyle\exp(-\beta A(z))dz , 
\end{equation} 
 where $A$ is the so-called free energy associated with the reaction coordinate $\xi$. By the co-formula  (see [\refcite{tony:08}], Appendix A), the following formula for the free energy can then be obtained: up to an additive constant,
\begin{equation}\label{free0}
A(z)=-\beta^{-1}\,{\rm ln}(Z_{\Sigma_{z}}),
\end{equation}
where $Z_{\Sigma_{z}}=\displaystyle\int_{\Sigma_{z}}e^{-\beta V(x)}\delta_{\xi(x)-z}(dx)$, the submanifold $\Sigma_{z}$ is defined by
 $$\Sigma_{z}=\{x=(x_1,...,x_n)\in\mathcal{D}\,|\,\xi(x)=z\},$$
  and $\delta_{\xi(x)-z}(dx)$ represents a measure with support $\Sigma_{z}$, such that $\delta_{\xi(x)-z}(dx)dz=dx$ (for further details on delta measures, we refer to [\refcite{tony:10}], Section~3.2.1). We assume henceforth that $\xi$ and $V$ are such that $Z_{\Sigma_{z}}<\infty$, for all $z\in \mathbb{R}^m$.

The idea of free energy biasing methods, such as the adaptive biasing force method (see [\refcite{dav:01,hen:04}]) or the Wang Landau algorithm (see [\refcite{wan:01}]), is that, if $\xi$ is well chosen, the dynamics associated to $V-A\circ\xi$ is less metastable than the dynamics associated to $V$. Indeed, from the definition of the free energy $\eqref{imxi}$, for any compact subspace $\mathcal{M}\subset \mathbb{R}^m$, the image of $\tilde{Z}^{-1}\mathrm{e}^{(-\beta(V-A\circ \xi)(x))}1_{\xi(x)\in\mathcal{M}}$ by $\xi$ is the uniform law $\displaystyle\frac{1_{\mathcal{M}}}{|\mathcal{M}|}$, where $\tilde{Z}=\displaystyle\int_{Z_{\Sigma_{z}}}\mathrm{e}^{(-\beta(V-A\circ \xi)(x))}1_{\xi(x)\in\mathcal{M}}$ and $|\mathcal{M}|$ denotes the Lebesgue measure on $\mathcal{M}$. The uniform law is typically easier to sample than the original measure $\xi*\mu$. If the function $\xi$ is well chosen (i.e. if the dynamics in the direction orthogonal to $\xi$ is not too metastable), the free energy can be used as a biasing potential to accelerate the sampling of the dynamics (see [\refcite{tony:08}]). The difficulty is of course that the free energy $A$ is unknown and difficult to approximate using the original dynamics \eqref{over} because of metastability. Actually, in many practical cases, it is the quantity of interest that one would like to approximate by molecular dynamics simulations (see [\refcite{Chi:07,tony:10}]). The principle of adaptive biasing methods is thus to approximate $A$ (or its gradient) on the fly in order to bias the dynamics and to reduce the metastable features of the original dynamics \eqref{over}.
 
\subsection{\textit{Adaptive biasing force method (ABF)}}
In order to introduce the ABF method, we need a formula for the derivatives of $A$. The so called {\em mean force} $\nabla A(z)$, can be obtained from \eqref{free0} as (see [\refcite{tony:10}], Section~3.2.2):
\begin{equation}\label{mean}
\nabla A(z)=\displaystyle \int_{\Sigma_{z}}f(x)d\mu_{\Sigma_{z}},
\end{equation}
where $d\mu_{\Sigma_{z}}$ is the probability measure $\mu$ conditioned to a fixed value $z$ of the reaction coordinate:
\begin{equation}\label{codms}
d\mu_{\Sigma_{z}}=Z_{\Sigma_{z}}^{-1}\mathrm{e}^{-\beta V(x)}\delta_{\xi(x)-z}(dx),
\end{equation}
and $f$ is the so-called {\em local mean force} defined by
\begin{equation}\label{cmf}
f_i=\displaystyle \sum _{j=1}^m G_{i,j}^{-1}\nabla \xi_j.\nabla V-\beta^{-1} {\rm div}\left(\sum _{j=1}^m G_{i,j}^{-1}\nabla \xi_j\right),
\end{equation}
where $G=(G_{i,j})_{i,j=1,...,m}$, has components $G_{i,j}=\nabla \xi_i\cdot\nabla \xi_j$. 
 This can be rewritten in terms of conditional expectation as: for a random variable $X$ with law $\mu$ (defined by \eqref{gib}),
\begin{equation}
\nabla A(z)=\mathbb{E}(f(X)|\xi(x)=z).
\end{equation}

We are now in position to introduce the standard adaptive biasing force (ABF) technique, applied to the overdamped Langevin dynamics \eqref{over}:   
  \begin{equation}\label{abf0}
    \left \lbrace
  \begin{aligned}
  dX_t &=- \left ( \nabla V-\sum_{i=1}^mF_t^i\circ\xi\nabla\xi_i+ \nabla(W\circ \xi) \right )(X_t) dt+\sqrt{2\beta^{-1}}dW_t,\\
F_t^i(z)&=\displaystyle\mathbb{E}[f_i(X_t) |\xi(X_t)=z],\,i=1,...,m,
  \end{aligned}
  \right.
\end{equation}	
where $f$ is defined in \eqref{cmf}. Compared with the original dynamics \eqref{over}, two modifications have been made to obtain the ABF dynamics \eqref{abf0}:
\begin{enumerate}
\item First and more importantly, the force $\displaystyle\sum_{i=1}^mF_t^i\circ\xi\nabla\xi_i$ has been added to the original force $-\nabla V$. At time $t$, $F_t$ approximates $\nabla A$ defined in \eqref{mean}.
\item Second, a potential $W\circ\xi$ has been added. This is actually needed in the case when $\xi$ lives in an unbounded domain. In this case, a so-called confining potential $W$ is introduced so that the law of $\xi(X_t)$ admits a longtime limit $Z_W^{-1}\mathrm{e}^{-\beta W(z)}dz$ (see Remark~\ref{conlaw} at the end of Section~\ref{pr}), where $Z_W=\displaystyle\int_{\mathrm{Ran}(\xi)}\mathrm{e}^{-\beta W}$ is assumed to be finite. When $\xi$ is living in a compact subspace of $\mathbb{R}^m$, there is no need to introduce such a potential and the law of $\xi(X_t)$ converges exponentially fast to the uniform law on the compact subspace (as explained in Section~\ref{meta} and Section~\ref{pr}). Typically, $W$ is zero in a chosen compact subspace $\mathcal{M}$ of $\mathbb{R}^m$ and is harmonic outside $\mathcal{M}$. For example, in dimension two, suppose that $\xi=(\xi_1,\xi_2)$ and $\mathcal{M}=[\xi_{min},\xi_{max}]\times[\xi_{min},\xi_{max}]$, then $W$ can be defined as: 
\begin{equation}\label{conf}
W(z_1,z_2)=\displaystyle\sum_{i=1}^2 1_{z_i\geq \xi_{max}}(z_i-\xi_{max})^2+\sum_{i=1}^21_{z_i\leq \xi_{min}}(z_i-\xi_{min})^2.
\end{equation}
\end{enumerate}
It is proven in [\refcite{tony:08}] that, under appropriate assumptions, $F_t$ converges exponentially fast to $\nabla A$. In addition, for well chosen $\xi$, the convergence to equilibrium for \eqref{abf0} is much quicker than for \eqref{over}. This can be quantified using entropy estimates and Logarithmic Sobolev Inequalities, see [\refcite{tony:08}].

Notice that even though $F_t$ converges to a gradient ($\nabla A$), there is no reason why $F_t$ would be a gradient at time $t$. In this paper, we propose an alternative method, where we approximate $\nabla A$, at any time $t$, by a gradient denoted $\nabla A_t$. The gradient $\nabla A_t$ is defined as the Helmholtz projection of $F_t$. One could expect improvements compared to the original ABF method since the variance of $\nabla A_t$ is then smaller than the variance of $F_t$ (since $A_t$ is a scalar function). Reducing the variance is important since the conditional expectation in \eqref{abf0} is approximated by empirical averages in practice.

\subsection{\textit{Projected adaptive biasing force method (PABF)}} 
A natural algorithm to reconstruct $A_t$ from $F_t$, consists in solving the following Poisson problem:
\begin{equation}\label{poi0}
\Delta A_t={\rm div}F_t\quad \mbox{on}\,\mathcal{M},
\end{equation}
with appropriate boundary conditions depending on the choice of $\xi$ and $\mathcal{M}$. More precisely, if $\xi$ is periodic and $\mathcal{M}$ is the torus $\mathbb{T}^m$, then we are working with periodic boundary conditions. If $\xi$ is with values in $\mathbb{R}^m$ and $\mathcal{M}$ is a bounded subset of $\mathbb{R}^m$, then Neumann boundary conditions are needed (see Remark \ref{neu} at the end of Section~\ref{conv}). To solve this Poisson problem, standard methods such as finite difference methods, finite element methods, spectral methods or Fourier transforms can be used. Note that \eqref{poi0} is the Euler equation associated with the minimization problem:
 \begin{equation}\label{min}
A_t=\displaystyle \mathop{\mbox{argmin}}_{g\in H^1(\mathcal{M})/\mathbb{R}}\int_{\mathcal{M}} |\nabla g-F_t|^2,
\end{equation}
where $\displaystyle H^1(\mathcal{M})/\mathbb{R}=\left\{ g\in H^1(\mathcal{M})\,|\,\int_{\mathcal{M}}g=0\right\}$ denotes the subspace of $H^1(\mathcal{M})$ of zero average functions.
In view of \eqref{min}, $A_t$ can be interpreted as the function such that its gradient is the closest to $F_t$.
Solving \eqref{poi0} amounts to computing the so-called {\em Helmholtz-Hodge} decomposition of the vector field $F_t$ as (see [\refcite{gira:86}], Section~3):
 \begin{equation}\label{helm0}
F_t=\nabla A_t+R_t,\quad \mbox{on}\,\mathcal{M},
\end{equation}
where $R_t$ is a divergence free vector field.

Finally, the {\em projected ABF dynamics} we  propose  to  study  is  the  following  non linear stochastic differential equation:
\begin{equation}\label{ABF}
\left\lbrace
\begin{aligned}
dX_t&=-\nabla(V-A_t\circ\xi+W\circ \xi)(X_t)dt+\!\sqrt{2\beta^{-1}}dW_t,\\
\Delta A_t&={\rm div}F_t\quad \mbox{on}\,\mathcal{M}, \mbox{ with appropriate boundary conditions},\\
F_t^i(z)&=\displaystyle\mathbb{E}[f_i(X_t)|\xi(X_t)=z],\,i=1,...,m.
\end{aligned}
\right.
\end{equation}	
Compared with the standard ABF dynamics \eqref{abf0}, the only modification is that the mean force $F_t$ is replaced by $\nabla A_t$, which is meant to be an approximation of $\nabla A$ at time $t$.

The main theoretical result of this paper is that $A_t$ converges exponentially fast to the free energy $A$ in $\mathcal{M}$ (at least in a specific setting and for a slightly modified version of \eqref{ABF}, see Section~\ref{lcpa} for more details). Moreover, we illustrate numerically this result on a typical example. From a numerical point of view, the interest of the method is that the variance of the projected estimated mean force (i.e. $\nabla A_t$) is smaller than the variance of the estimated mean force (i.e. $F_t$). We observe numerically that this variance reduction enables a faster convergence to equilibrium for PABF compared with the original ABF.

The paper is organized as follows. In Section~\ref{lcpa}, the longtime convergence of the projected ABF method is proven. Section~\ref{num} is devoted to a numerical illustration of the interest of the projected ABF compared to the standard ABF approach. Finally, the proofs of the results presented in Section~\ref{lcpa} are provided in Section~\ref{pro}.

 \section{Longtime convergence of the projected ABF method}\label{lcpa}
 For the sake of simplicity, we assume in this Section that $\mathcal{D}=\mathbb{T}^n$ and that $\xi(x)=(x_1,x_2)$. Then $\xi$ lives in the compact space $\mathcal{M}=\mathbb{T}^2$ and we therefore take $W=0$. The free energy can be written as: 
\begin{equation}\label{free}
A(x_1,x_2)=-\beta^{-1}\,{\rm ln}(Z_{\Sigma_{(x_1,x_2)}}),
\end{equation}
where $Z_{\Sigma_{(x_1,x_2)}}=\displaystyle\int_{\Sigma_{(x_1,x_2)}}e^{-\beta V(x)}dx_3...dx_n$ and $\Sigma_{(x_1,x_2)}=\{x_1,x_2\}\times \mathbb{T}^{n-2}$. The mean force becomes:
\begin{equation}\label{meanf}
\nabla A(x_1,x_2)=\displaystyle \int_{\Sigma_{(x_1,x_2)}}f(x)d\mu_{\Sigma_{(x_1,x_2)}},
\end{equation}
where $f=(f_1,f_2)=(\partial_1 V,\partial_2 V)$ and the conditional probability measure $d\mu_{\Sigma_{(x_1,x_2)}}$ is:
$$d\mu_{\Sigma_{(x_1,x_2)}}=Z_{\Sigma_{(x_1,x_2)}}^{-1} {\rm e}^{-\beta V}dx_3...dx_n.$$
 Finally, the vector field $F_t(x_1,x_2)$ writes $\displaystyle \int_{\Sigma_{(x_1,x_2)}}fd\mu_{\Sigma_{(x_1,x_2)}}(t,.)$, or equivalently: 
 $$F_t^i(x_1,x_2)=\mathbb{E}\big(\partial_iV(X_t)|\xi(X_t)=(x_1,x_2)\big),\,i=1,2.$$

 \subsection{Helmholtz projection}\label{he}
In section \ref{heldec}, weighted Helmholtz-Hodge decomposition of $F_t$ is presented. In section \ref{minpb}, the associated minimization problem and projection operator are introduced.

  Let us first fix some notations. For $x\in\mathbb{R}^n$ and $1\leq i<j\leq n$, $x_i^j$ denotes the vector $(x_i,x_{i+1},...,x_j)$ and $dx_i^j$ denotes $dx_i\,dx_{i+1}\,...\,dx_j$. Moreover, $\nabla_{x_1^2}$, ${\rm div}_{x_1^2}$ and $\Delta_{x_1^2}$ represent respectively the gradient, the divergence and the laplacian in dimension two for the first two variables $(x_1,x_2)$. Likewise, $\nabla_{x_3^n}=(\partial_{3},...,\partial_n)^T$ represents the gradient vector starting from the third variable of $\mathbb{R}^n$.
  \subsubsection{\textit{Helmholtz decomposition}}\label{heldec}
The space $\mathbb{T}^2$ is a bounded and connected space. For any smooth positive probability density function $\varphi:\mathbb{T}^2\rightarrow \mathbb{R}$, let us define the weighted Hilbert space: $L^2_{\varphi}(\mathbb{T}^2)=\{f:\mathbb{T}^2\rightarrow \mathbb{R},\,\int_{\mathbb{T}^2}|f|^2\varphi<\infty\}$. Let us also introduce the Hilbert space $H_{\varphi}({\rm div};\mathbb{T}^2)=\{g\in L^2_{\varphi}(\mathbb{T}^2)\times L^2_{\varphi}(\mathbb{T}^2),\,{\rm div}_{x_1^2}( g)\in L^2(\mathbb{T}^2)\}$. It is well-known that any vector field $F_t:\mathbb{T}^2\rightarrow \mathbb{R}^2$ $\in\,H_1({\rm div},\mathbb{T}^2)$ can be written (see [\refcite{gira:86}], Section~3 for example) as (Helmholtz decomposition): $F_t=\nabla_{x_1^2} A_t+R_t$, where $R_t$ is a divergence free vector field. We will need a generalization of the standard Helmholtz decomposition to the weighted Hilbert spaces $L^2_{\varphi}(\mathbb{T}^2)$ and $H_{\varphi}({\rm div};\mathbb{T}^2))$:
\begin{equation}\label{hel}
F_t\varphi=\nabla_{x_1^2} (A_t)\varphi+R_t,
\end{equation}
s.t. ${\rm div}_{x_1^2}(R_t)=0$. This weighted Helmholtz decomposition is required to simplify calculations when studying the longtime convergence (see Remark~\ref{simpdi} in Section~\ref{difprf} for more details). Recall the space: $H^1(\mathbb{T}^2)/\mathbb{R}=\left\{ g\in H^1(\mathbb{T}^2)\,|\,\int_{\mathbb{T}^2}g=0\right\}$. The function $A_t$ is then the solution to the following problem: 
\begin{equation}\label{sca}
\displaystyle\int_{\mathbb{T}^2}\nabla_{x_1^2} A_t\cdot\nabla_{x_1^2} g\,\varphi=\int_{\mathbb{T}^2}F_t\cdot\nabla_{x_1^2} g\,\varphi,\quad \forall g\in H^1(\mathbb{T}^2)/\mathbb{R},
\end{equation}
which is the weak formulation of the Poisson problem:
\begin{equation} \label{poi}
{\rm div}_{x_1^2}(\nabla_{x_1^2} A_t\varphi(t,.))={\rm div}_{x_1^2}(F_t\varphi(t,.)),
\end{equation}
with periodic boundary conditions. Using standard arguments (Lax-Milgram theorem), it is straight forward to check that \eqref{sca} admits a unique solution $A_t\in H^1(\mathbb{T}^2)/\mathbb{R}$. 

  \subsubsection{\textit{Minimization problem and projection on a gradient}}\label{minpb}
\begin{proposition}\label{propmin}
Suppose that $F_t \in H_{\varphi}({\rm div}; \mathbb{T}^2)$. Then for any smooth positive probability density function $\varphi$, the equation \eqref{sca} is the Euler Lagrange equation associated with the following minimization problem:\\
 \begin{equation}\label{minw}
A_t=\displaystyle \min_{h\in  H^1(\mathbb{T}^2)/\mathbb{R}}\int_{\mathbb{T}^2} |\nabla_{x_1^2} h-F_t|^2\varphi= \min_{h\in  H^1(\mathbb{T}^2)/\mathbb{R}}||\nabla_{x_1^2} h-F_t||^2_{L^2_{\varphi}(\mathbb{T}^2)}.
\end{equation}
Furthermore, $A_t$ belongs to $H^2(\mathbb{T}^2)$.
\end{proposition}

\begin{proof}
Let us introduce the application $I:H^1(\mathbb{T}^2)/\mathbb{R}\rightarrow \mathbb{R}^+$, defined by $I(g)=||\nabla_{x_1^2} h-F_t||^2_{L^2_{\varphi}(\mathbb{T}^2)}$. It is easy to prove that $I$ is $\alpha$-convex and coercive, i.e. $\displaystyle\lim_{\|g\|_{H^1}\!\rightarrow+\infty}I(g)=+\infty$. Thus $I$ admits a unique global minimum $A_t\in \displaystyle H^1(\mathbb{T}^2)/\mathbb{R}$. Furthermore, $\forall\varepsilon>0,\forall g\in H^1(\mathbb{T}^2)$,
 \begin{equation}\label{epsl}
\begin{aligned}
 I(A_t+\varepsilon g)&=\displaystyle\int_{\mathbb{T}^2}|\nabla_{x_1^2}( A_t+\varepsilon g)-F_t|^2\varphi\\
 &=\displaystyle\int_{\mathbb{T}^2}|\nabla_{x_1^2} A_t-F_t|^2\varphi-2\varepsilon\int_{\mathbb{T}^2} (\nabla_{x_1^2} A_t-F_t)\cdot\nabla_{x_1^2} g\,\varphi+\varepsilon^2\int_{\mathbb{T}^2}|\nabla_{x_1^2} g|^2\varphi\\
 &=I(A_t)-2\varepsilon\displaystyle\int_{\mathbb{T}^2}(\nabla_{x_1^2} A_t-F_t)\cdot\nabla_{x_1^2} g\,\varphi+\varepsilon^2\int_{\mathbb{T}^2}|\nabla_{x_1^2} g|^2\varphi.
\end{aligned}
\end{equation}

Since $A_t$ is the minimum of $I$, then $I(A_t+\varepsilon g)\geq I(A_t),\,\forall\varepsilon>0,\forall g\in H^1(\mathbb{T}^2)$. By considering the asymptotic regime $\varepsilon\rightarrow 0$ in the last equation, one thus obtains the equation \eqref{sca}:
 \begin{equation*}
\displaystyle\int_{\mathbb{T}^2}\nabla_{x_1^2} A_t\cdot\nabla_{x_1^2} g\,\varphi=\int_{\mathbb{T}^2} F_t\cdot\nabla_{x_1^2} g\,\varphi,\,\forall g\in H^1(\mathbb{T}^2)/\mathbb{R}.
\end{equation*}
This is the weak formulation of the following problem \eqref{poi} in $H^1(\mathbb{T}^2)/\mathbb{R}$:
$${\rm div}_{x_1^2}(\nabla_{x_1^2} A_t\varphi(t,.))={\rm div}_{x_1^2}(F_t\varphi(t,.)).$$
Since $\varphi$ is a smooth positive function, then $\exists \,\delta>0$, s.t. $\varphi>\delta$. Furthermore, since ${\rm div}_{x_1^2}(F_t\varphi(t,.))\in L^2(\mathbb{T}^2)$, thus $\Delta_{x_1^2} A_t\in L^2(\mathbb{T}^2).$ Therefore, using standard elliptic regularity results, $A_t\in H^2(\mathbb{T}^2)$. 
\end{proof}

For any positive probability density function $\varphi$, the estimation vector field $\nabla_{x_1^2}A_t$ is the projection of $F_t$ onto a gradient. In the following, we will use the notation:
\begin{equation}\label{projop}
\mathcal{P}_{\varphi} (F_t)=\nabla_{x_1^2}A_t,
\end{equation}
 where the projection operator $\mathcal{P}_{\varphi}$ is a linear projection defined from $H_{\varphi}({\rm div}; \mathbb{T}^2)$ to $ H^1(\mathbb{T}^2)\times H^1(\mathbb{T}^2)$. Notice in particular that $\mathcal{P}_{\varphi}\circ\mathcal{P}_{\varphi}=\mathcal{P}_{\varphi}$. 

\subsection{The projected ABF (PABF) method}\label{pr}
We will study the longtime convergence of the following PABF dynamics:
\begin{equation}\label{pabf}
\left\lbrace
\begin{aligned}
 dX_t&=-\nabla\big( V-A_t\circ\xi\big)(X_t)dt+\sqrt{2\beta^{-1}}dW_t,\\
\nabla_{x_1^2} A_t&=\mathcal{P}_{\psi^{\xi}} (F_t),\\
 F_t^i(x_1,x_2)&=\displaystyle\mathbb{E}[\partial_i V(X_t)|\xi(X_t)=(x_1,x_2)],\,i=1,2,
\end{aligned}
\right.
\end{equation}	
  where $\mathcal{P}_{\psi^{\xi}} $ is the linear projection defined by \eqref{projop} and $W_t$ is a standard $nN$-dimensional Brownian motion. Thanks to the diffusion term $\sqrt{2\beta^{-1}}dW_t$, $X_t$ admits a smooth density $\psi$ with respect to the the Lebesgue measure on $\mathbb{T}^n$ and $\psi^{\xi}$ then denotes the marginal distribution of $\psi$ along $\xi$:
  \begin{equation}\label{psixi}
\psi^{\xi}(t,x_1,x_2)=\displaystyle\int_{\Sigma_{(x_1,x_2)}}\psi(t,x)dx_3^n.\\
\end{equation}
The dynamics \eqref{pabf} is the PABF dynamics \eqref{ABF} with $\xi(x)=(x_1,x_2)$, $W=0$ and a weighted Helmholtz projection. The weight $\psi^{\xi}$ is introduced to simplify the convergence proof (see Remark~\ref{simpdi} in Section~\ref{pro}).
    \begin{remark}\label{condme}
  If the law of $X_t$ is $\psi(t,x)dx$ then the law of $\xi(X_t)$ is $\psi^{\xi}(t,x_1,x_2)dx_1dx_2$ and the conditional distribution of $X_t$ given $\xi(X_t)=(x_1,x_2)$ is (see \eqref{codms} for a similar formula when $\psi=Z^{-1}_{\Sigma_{(x_1,x_2)}} \mathrm{e}^{-\beta V}$):\\
\begin{equation}\label{cond}
d\mu_{t,x_1,x_2}=\displaystyle\frac{\psi(t,x)dx_3^n}{\psi^{\xi}(t,x_1,x_2)}.
\end{equation}
Indeed, for any smooth functions $f$ and $g$,

$\begin{aligned}
\mathbb{E}(f(\xi(X_t))g(X_t))&=\displaystyle \int_{\mathbb{T}^n}f(\xi(x))g(x)\psi(t,x)dx\\
&=\int_{\mathbb{T}^2}\int_{\Sigma_{(x_1,x_2)}}f\circ \xi g\psi dx_3^ndx_1dx_2\\
&=\int_{\mathbb{T}^2}f(x_1,x_2)\frac{\int_{\Sigma_{(x_1,x_2)}} g\psi dx_3^n}{\psi^{\xi}(x_1,x_2)}\psi^{\xi}(x_1,x_2)dx_1dx_2.
\end{aligned}$\\
$\lozenge$
  \end{remark}
  
Let us now introduce the non linear partial differential equation (the so-called {\em Fokker-Planck equation}) which rules the evolution of the density $\psi(t,x)$ of $X_t$ solution of \eqref{pabf}:\\
\begin{equation}\label{Fok}
\left\{						
\begin{array}{l}
\displaystyle\partial_t\psi={\rm div}\left (\left(\nabla
\displaystyle V-\sum_{i=1}^2\partial_iA_t\circ\xi\nabla\xi_i\right )\psi+\beta^{-1}\nabla\psi\right ),\quad \mbox{for}\, (t,x)\in[0,\infty[\times \mathbb{T}^n,\\
\forall t\geq 0,\,{\rm div}(\nabla A_t\psi^{\xi}(t,.))={\rm div}(F_t\psi^{\xi}(t,.)),\, \mbox{in }\mathbb{T}^2\mbox{ with periodic boundary conditions,}\\
\displaystyle \forall t\geq 0,\,\forall (x_1,x_2)\in \mathbb{T}^2,\, F_t^i(x_1,x_2)=\frac{\displaystyle\int_{\Sigma_{(x_1,x_2)}}\partial_iV\psi dx_3^n}{\psi^{\xi}(x_1,x_2)},\,i=1,2.\\
\end{array}
\right.
\end{equation}
The first equation of \eqref{Fok} rewrites:
 \begin{equation}\label{reform}
 \partial_t\psi={\rm div}[\nabla V\psi+\beta^{-1}\nabla\psi]-\partial_1((\partial_1A_t)\psi)-\partial_2((\partial_2A_t)\psi).
 \end{equation}
Suppose that ($\psi,F_t,A_t)$ is a solution of \eqref{Fok} and let us introduce the expected long-time limits of $\psi$, $\psi^{\xi}$ (defined by \eqref{psixi}) and $\mu_{t,x_1,x_2}$ (defined by \eqref{cond}) respectively:
\begin{enumerate}
  \item $\psi_{\infty}=e^{-\beta(V-A\circ\xi)}$;
  \item $\psi_{\infty}^{\xi}=1$ (uniform law);
  \item \begin{equation}\label{ecm}
   \mu_{\infty,x_1,x_2}=Z_{\Sigma_{(x_1,x_2)}}^{-1}e^{-\beta V}dx_3^n.
 \end{equation}

\end{enumerate}
Notice that the probability measure $\psi_{\infty}^{\xi}(x_1,x_2)dx_1dx_2$ is the image of the probability measure $\psi_{\infty}(x)dx$ by $\xi$ and that $ \mu_{\infty,x_1,x_2}=\mu_{\Sigma_{(x_1,x_2)}}$ defined in \eqref{codms}. Furthermore, we have that 
 $$\displaystyle\int_{\mathbb{T}^n}\psi_{\infty}=1,\displaystyle\int_{\mathbb{T}^2}\psi_{\infty}^{\xi}=1 \mbox{ and }\forall (x_1,x_2)\in \mathbb{T}^2,\quad\displaystyle\int_{\Sigma_{(x_1,x_2)}}d\mu_{\infty,x_1x_2}=1.$$

\begin{remark}\label{conlaw}
In the case when $W\neq0$, the first equation of the Fokker-Plank problem~\eqref{Fok} becomes:
$$\partial_t\psi={\rm div}\left (\nabla\left(V-A_t\circ\xi-W\circ\xi \right)\psi+\beta^{-1}\nabla\psi\right ).$$
The expected long-time limits of $\psi$, $\psi^{\xi}$ and $\mu_{t,x_1,x_2}$ are respectively:
\begin{enumerate}
  \item $\psi_{\infty}=Z_W^{-1}e^{-\beta(V-A\circ\xi-W\circ\xi)}$;
  \item $\psi_{\infty}^{\xi}=Z_W^{-1}e^{-\beta W}$;
  \item $\mu_{\infty,x_1,x_2}=Z_{\Sigma_{(x_1,x_2)}}^{-1}e^{-\beta V}dx_3^n$,
   \end{enumerate}
   where $\displaystyle Z_W=\int_{\mathbb{T}^2}e^{-\beta W}$.
\end{remark}
\subsection{Precise statements of the longtime convergence results}\label{res}
In section~\ref{entr}, some well-known results on entropy techniques are presented. For a general introduction to logarithmic Sobolev inequalities, their properties and their relation to long-time behaviours of solutions to partial differential equations, we refer to [\refcite{ane:00,arn:01,Vil:03}]. Section~\ref{conv} presents the main theorem of convergence.

 \subsubsection{\textit{Entropy and Fisher information}}\label{entr}

  Define the relative entropy $H(.|.)$ as follows: for any probability measures $\mu$ and $\nu$ such that $\mu$ is absolutely continuous with respect to $\nu$ (denoted $\mu \ll\nu$),
  $$H(\mu | \nu)=\displaystyle\int{\rm ln}\left(\frac{d\mu}{d\nu}\right)d\mu.$$
Abusing the notation, we will denote $H(\varphi|\psi)$ for $H(\varphi(x) dx|\psi(x) dx)$ in case of probability measures with densities.
Let us recall the Csiszar-Kullback inequality (see [\refcite{arn:01}]):
\begin{equation}\label{kul}
\|\mu-\nu\|_{TV}\leq \sqrt{2H(\mu|\nu)},
\end{equation}
where $\|\mu-\nu\|_{TV}=\displaystyle\sup_{\|f\|_{L^{\infty}}\leq1}\left\{\int fd(\mu-\nu)\right\}$ is the total variation norm of the signed measure $\mu-\nu$. When both $\mu$ and $\nu$ have densities with respect to the Lebesque measure, $\|\mu-\nu\|_{TV}$ is simply the $L^1$ norm of the difference between the two densities. The entropy $H(\mu|\nu)$ can be understood as a measure of how close $\mu$ and $\nu$ are.

Now, let us define the Fisher information of $\mu$ with respect to $\nu$:
\begin{equation}\label{fish}
I(\mu | \nu)=\displaystyle\int\left|\nabla{\rm ln}\left(\frac{d\mu}{d\nu}\right)\right|^2d\mu.
\end{equation}

The Wasserstein distance is another way to compare two probability measures $\mu$ and $\nu$ defined on a space $\Sigma$,
$$W(\mu,\nu)=\sqrt{\displaystyle\inf_{\pi\in\prod(\mu,\nu)}\int_{\Sigma\times\Sigma}d_{\Sigma}(x,y)^2d\pi(x,y)},$$
where the geodesic distance  $d_{\Sigma}$ on $\Sigma$ is defined as: $\forall \, x,y\in \Sigma,$
$$d_{\Sigma}(x,y)=\inf\left\{\sqrt{\int_0^1|\dot{w}(t)|^2dt}\,\mid \,w\in C^1([0,1],\Sigma),\,w(0)=x,w(1)=y\right\},$$
and $\prod(\mu,\nu)$ denotes the set of coupling probability measures, namely probability measures on $\Sigma\times\Sigma$ such that their marginals are $\mu$ and $\nu$: $\forall\pi\in\prod(\mu,\nu),\int_{\Sigma\times\Sigma}\phi(x)d\pi(x,y)=\int_{\Sigma}\phi d\mu $ and $\int_{\Sigma\times\Sigma}\psi(y)d\pi(x,y)=\int_{\Sigma}\psi d\nu .$

\begin{definition}
We say that a probability measure $\nu$ satisfies a logarithmic Sobolev inequality with constant $\rho>0$ (denoted LSI($\rho$)) if for all probability measure $\mu$ such that $\mu \ll\nu$ ,
$$H(\mu | \nu)\leq\frac{1}{2\rho}I(\mu | \nu).$$
\end{definition}

\begin{definition}
We say that a probability measure $\nu$ satisfies  a Talagrand inequality with constant $\rho>0$ (denoted T($\rho$)) if for all probability measure $\mu$ such that $\mu \ll\nu$,
$$W(\mu,\nu)\leq\sqrt{\frac{2}{\rho}H(\mu | \nu)}.$$
\end{definition}
\begin{remark}
We implicitly assume in the latter definition, that the probability measures have finite moments of order 2. This is the case for the probability measures used in this paper.
\end{remark}
The following lemma is proved in [\refcite{Ott:00}], Theorem~1:
\begin{lemma}\label{tala}
If $\nu$ satisfies $LSI(\rho)$, then $\nu$ satisfies $T(\rho)$. 
\end{lemma}
Recall that $X_t$ solution to \eqref{pabf} has a density $\psi(t,.)$. In the following, we denote the {\em Total Entropy} by 
\begin{equation}\label{tent}
 E(t)=H(\psi(t,.)|\psi_{\infty})=\int_{\mathbb{T}^n}\ln(\psi/\psi_{\infty})\psi ,
\end{equation}
the {\em Macroscopic Entropy} by
\begin{equation}\label{maent}
E_M(t)=H(\psi^{\xi}(t,.)|\psi^{\xi}_{\infty})=\int_{\mathbb{T}^2}\ln(\psi^{\xi}/\psi^{\xi}_{\infty})\psi^{\xi},
\end{equation}
and the {\em Microscopic Entropy} by 
\begin{equation}\label{mient}
E_m(t)=\displaystyle\int_{\mathcal{M}}e_m(t,x_1,x_2)\psi^{\xi}(t,x_1,x_2)dx_1dx_2,
\end{equation}
 where $e_m(t,x_1,x_2)=H(\mu_{t,x_1,x_2}|\mu_{\infty,x_1,x_2})$. The following result is straightforward to check:

\begin{lemma}\label{entsum} It holds, $\forall t\geq 0$,
$$E(t)=E_M(t)+E_m(t).$$
\end{lemma}
Note that the Fisher information of $\mu_{t,x_1,x_2}$ with respect to $\mu_{\infty,x_1,x_2}$ can be written as (see \eqref{fish}): 
$$I(\mu_{t,x_1,x_2}|\mu_{\infty,x_1,x_2})=\displaystyle \int_{\Sigma_{(x_1,x_2)}}|\nabla_{x_3^n}\ln(\psi(t,.)/\psi_{\infty})|^2d\mu_{t,x_1,x_2}.$$

\subsubsection{\textit{Convergence of the PABF dynamics \eqref{Fok}}}\label{conv}
The following proposition shows that the density function $\psi^{\xi}$ satisfies a simple diffusion equation.
  \begin{proposition}\label{diff}
Suppose that $(\psi,F_t,A_t)$ is a smooth solution of \eqref{Fok}. Then $\psi^{\xi}$ satisfies:
\begin{equation}\label{dif2}
\left\{\begin{array}{rl}
\partial_t\psi^{\xi}&=\beta^{-1}\Delta_{x_1^2}\psi^{\xi}, \mbox{ in }[0,\infty[\times \mathbb{T}^2,\\
\psi^{\xi}(0,.)&=\psi^{\xi}_0,\mbox{ on } \mathbb{T}^2.
\end{array}
\right.
\end{equation}
\end{proposition}

\begin{remark}\label{rmdif}
If $ \psi^{\xi}_0= 0$ at some points or is not smooth, then $F$ at time $0$ may not be well defined or $I(\psi^{\xi}(0,.)/\psi^{\xi}_{\infty})$ may be infinite. Since, by Proposition \ref{diff}, $ \psi^{\xi}$ satisfies a simple diffusion equation these difficulties disappear as soon as $t > 0$. Therefore, up to considering the problem for $t>t_0>0$, we can suppose that $ \psi^{\xi}_0$ is a smooth positive function.
We also have that for all $t>0$, $\psi^{\xi}(t,.)>0$, $\displaystyle\int_{\mathbb{T}^2}\psi^{\xi}=1$ and $\psi^{\xi}(t,.)\in C^{\infty}(\mathbb{T}^2).$
\end{remark}

\begin{remark}
In the case where $W\neq 0$, the probability density function $\psi^{\xi}$ satisfies the modified diffusion equation:
$$\partial_t\psi^{\xi}=\nabla_{x_1^2}\cdot \left ( \beta^{-1}\nabla_{x_1^2} \psi^{\xi}+\psi^{\xi}\nabla_{x_1^2}W \right).$$
\end{remark}
Here are two simple corollaries of Proposition~\ref{diff}.
\begin{corollary}\label{fishco}
There exists $t_0>0$ and $I_0>0$ (depending on $\psi^{\xi}_0$), such that
$$\forall t>t_0,\quad I(\psi^{\xi}(t,.)|\psi^{\xi}_{\infty})<I_0e^{-\beta^{-1}8\pi^2t}.$$
\end{corollary}

\begin{corollary}\label{mic}
The macroscopic entropy $E_M(t)$, defined by \eqref{maent}, converges exponentially fast to zero: 
$$\forall t>t_0,\quad  E_M(t)\leq \frac{I_0}{8\pi^2}e^{-\beta^{-1}8\pi^2t},$$
where $I_0$ is the constant introduced in Corollary~\ref{fishco}.
\end{corollary}

The assumptions we need to prove the longtime convergence of the biasing force $\nabla A_t$ to the mean force $\nabla A$ are the following:
\begin{description}
  \item[[H1\!\!]] $V\in C^2(\mathbb{T}^n)$ and satisfies:
  $$\exists \gamma >0,\,\forall \,3\leq j\leq n,\forall x\in \mathbb{T}^n,\quad \max(|\partial_1\partial_j V(x)|,|\partial_2\partial_j V(x)|)\leq \gamma .$$
  \item[[H2\!\!]] $V$ is such that $\exists\rho>0$, $\forall (x_1,x_2)\in  \mathbb{T}^2$, $\mu_{\infty,x_1,x_2}=\mu_{\Sigma(x_1,x_2)}$ defined by \eqref{ecm} satisfies $LSI(\rho).$
\end{description}
The main theorem is:
\begin{theorem}\label{theo1}
Let us assume [H1] and [H2]. The following properties then hold:
\begin{enumerate}
  \item The microscopic entropy $E_m$ converges exponentially fast to zero: 
  \begin{equation}\label{Em}
  \exists C>0,\exists \lambda>0,\forall t\geq 0,\quad \sqrt{E_m(t)}\leq Ce^{-\lambda t}.
  \end{equation}
  Furthermore, if $\rho\neq 4\pi^2$, then $\lambda=\beta^{-1}\min (\rho,4\pi^2)$ and\\
   $\displaystyle C=\sqrt{E_m(0)}+\frac{\gamma}{\beta^{-1}|\rho-4\pi^2|}\sqrt{\frac{I_0}{2\rho}}$. 
   If $\rho=4\pi^2$, then for all $\lambda<\beta^{-1}\rho$, there exists a positive constant $C$ such that \eqref{Em} is satisfied.
  \item $\sqrt{E(t)}$ and $\|\psi(t,.)-\psi_{\infty}\|_{L^1( \mathbb{T}^n)}$ both converge exponentially fast to zero with rate $\lambda$.
  \item The biasing force $\nabla_{x_1^2} A_t$ converges to the mean force $\nabla_{x_1^2} A$ in the following sense:
    \begin{equation}\label{At}
  \displaystyle\forall\, t\geq 0,\,\int_{\mathbb{T}^2}|\nabla_{x_1^2} A_t-\nabla_{x_1^2} A|^2\psi^{\xi}(t,x_1,x_2)dx_1dx_2\leq\frac{8\gamma^2}{\rho}E_m(t).
  \end{equation}
\end{enumerate}
\end{theorem}
The proofs of the results presented in this section are provided in Section~\ref{pro}.
   \begin{remark}
We would like to emphasize that our arguments hold under the assumption of existence of regular solutions. In particular, we suppose that the density $\psi(t,.)$ is sufficiently regular so that the algebric manipulations in the proofs (see Section~\ref{pro}) are valid.
\end{remark}
 \begin{remark}
 This remark is devoted to show how the rate of convergence of the dynamics \eqref{over} is improved thanks to PABF method. First of all, we mention a classical computation to get a rate of convergence for \eqref{over}. Precisely, if one denotes $\varphi (t,.)$ the probability density function of $X_t$ satisfying \eqref{over}, and $\varphi_{\infty}=Z_{\mu}^{-1}\mathrm{e}^{-\beta V}$ its longtime limit, then by standard computations (see for example [\refcite{arn:01}]), one obtains: 
 $$ \frac{d}{dt}H(\varphi(t,.)|\varphi_{\infty}) =-\beta^{-1}I(\varphi(t,.)|\varphi_{\infty}).$$
 Therefore, if $\varphi_{\infty}$ satisfies $LSI(R)$, then one obtains the estimate
 \begin{equation}\label{overcv}
 \exists R>0,\,\forall t>0,\quad H(\varphi(t,.)|\varphi_{\infty})\leq H(\varphi_0|\varphi_{\infty})\,\mathrm{e}^{-2\beta^{-1}Rt}.
  \end{equation}
From \eqref{kul}, we obtain that $\|\varphi(t,.)-\varphi_{\infty}\|_{L^1( \mathbb{T}^n)}$ converges exponentially fast to zero with rate $\beta^{-1}R$. The constant $R$ is known to be small if the metastable states are separated by large energy barriers or if high probability regions for $\mu$ are separated by large regions with small probability (namely $\mu$ is a multimodal measure).
Second, by Theorem~\ref{theo1}, one can show that $\nabla A_t$ converges exponentially fast to $\nabla A$ in $L^2(\psi^{\xi}_{\infty}(x_1,x_2)dx_1dx_2)-$norm at rate $\lambda=\beta^{-1}\min (\rho,4\pi^2)$. Indeed, since $\psi^{\xi}_{\infty}=1$,
\begin{align*}
\int_{\mathbb{T}^2}|\nabla A_t-\nabla A|^2dx_1dx_2&=\int_{\mathbb{T}^2}|\nabla A_t-\nabla A|^2\frac{\psi^{\xi}(t,x_1,x_2)}{\psi^{\xi}(t,x_1,x_2)}dx_1dx_2\\
&\leq \frac{8\gamma^2}{(1-\varepsilon)\rho}E_m(t),
\end{align*}
where $\varepsilon >0$ such that $\psi^{\xi}(t,x_1,x_2)\geq 1-\varepsilon$ (for more details refer to the proof of Corollary~\ref{fishco} in Section~\ref{pro}).
 This result must be compared with \eqref{overcv}. More precisely, $\lambda$ is related to the rate of convergence $4\pi^2$ at the macroscopic level, for equation \eqref{dif2} satisfied by $\psi^{\xi}$, and the rate of convergence $\rho$ at the microscopic level, coming from the logarithmic Sobolev inequalities satisfied by the conditional measures $\mu_{\infty,x_1,x_2}$. Of course, $\rho$ depends on the choice of the reaction coordinate. In our framework, we could state that a "good reaction coordinate" is such that $\rho$ is as large as possible. Typically, for good choices of $\xi$, $\lambda \gg R$, and the PABF dynamics converges to equilibrium much faster than the original dynamics \eqref{over}. This is typically the case if the conditional measures $\mu_{\infty,x_1,x_2}$ are less multimodal than the original measure $\mu$.
 \end{remark}

\begin{remark}(Extension to other geometric setting)\label{neu}\\
The results of Theorem~\ref{theo1} are easily generalized to the following setting: 

If $\mathcal{D}=\mathbb{R}^n$, $\xi(x)=(x_1,x_2)$ and $\mathcal{M}$ is a compact subspace of $\mathbb{R}^n$, then choose a confining potential $W$ (defined in \eqref{conf}) such that $Z_{W}=\int \mathrm{e}^{-\beta W}<+\infty$, $Z_{W}^{-1}\mathrm{e}^{-\beta W}$ satisfies $LSI(r^*)$ (for some $r^*>0$) and $W$ is convex potential, then Corollary~\ref{fishco} is satisfied with rate $2\beta^{-1}(r^*-\varepsilon)$, for any $\varepsilon\in (0,r^*)$ (refer to Corollary~1 in [\refcite{tony:08}] for further details). In this case, Neumann boundary conditions are needed to solve the Poisson problem \eqref{poi}:
  \begin{equation}\label{poisext}
\left\lbrace
\begin{aligned}
{\rm div}(\nabla A_t\psi^{\xi}(t,.))&={\rm div}(F_t\psi^{\xi}(t,.))& \mbox{in}\,\mathcal{M},\\
\frac{\partial A_t}{\partial n}&=F_t.n &\text{on}\,\partial \mathcal{M},
\end{aligned}
\right.
\end{equation}	 
where $n$ denotes the unit normal outward to $ \mathcal{M}$. The convergence rate $\lambda$ of Theorem~\ref{theo1} becomes $\beta^{-1}\min(\rho,r^*-\varepsilon)$. Neumann boundary conditions come from the minimization problem \eqref{minw} associated to the Euler-Lagrange equation. The numerical applications in Section~\ref{num} are performed in this setting.

\end{remark}
\begin{remark}(Extension to more general reaction coordinates)\\
In this section, we have chosen $\xi(x_1,...,x_n)=(x_1,x_2)$. The results can be extended to the following settings: 
\begin{enumerate}
\item In dimension one, the Helmholtz projection has obviously no sense. If $\mathcal{D}=\mathbb{T}^n$ and $\xi(x)=x_1$, then $F_t$ converges to $A'$, which is a derivative of a periodic function and thus $\displaystyle\int_{\mathbb{T}}A'=0$. Since $\displaystyle\int_{\mathbb{T}}F_t$ is not necessary equal to zero, one can therefore take a new approximation $A_t^{'}=F_t-\displaystyle \int_{\mathbb{T}}F_t$, which approximates $A'$ at any time $t$. The convergence results of this section can be extended to this setting, to show that $A'_t$ converges exponentially fast to $A'$.

\item More generally, for a reaction coordinate with values in $\mathbb{T}^m$, the convergence results presented in this paper still hold under the following orthogonality condition:
\begin{equation}\label{ortho}
\forall i\neq j,\,\nabla \xi_i\cdot\nabla \xi_j=0.
\end{equation}
In the case when \eqref{ortho} does not hold, it is possible to resort to the following trick used for example in metadynamics (refer to [\refcite{bu:06,tony:07}]). The idea is to introduce an additional variable $z$ of dimension $m$, and an extended potential $V_{\xi}(x,z)=V(x)+\frac{\kappa}{2}|z-\xi(x)|^2$, where $\kappa$ is a penalty constant. The reaction coordinate is then chosen as $\xi_{meta}(x,z)=z$, so that the associated free energy is:
$$A_{\xi}(z)=-\beta^{-1}\displaystyle \ln \int_{\mathcal{D}}e^{-\beta V_{\xi}(x,z)}dx,$$
which converges to $A(z)$ (defined in \eqref{free0}) when $\kappa$ goes to infinity. The extended PABF dynamics can be written as:
$$
\left\lbrace
\begin{aligned}
dX_t &= - \left( \nabla V(X_t) + \kappa \sum_{i=1}^m (\xi_i(X_t) - Z_{i,t}) \nabla \xi_i (X_t) \right) \, dt + \sqrt{2 \beta^{-1}} dW_t,\\
dZ_t &= \kappa (\xi(X_t) - \nabla E_t (Z_t)) \, dt + \sqrt{2 \beta^{-1}} d \overline{W}_t,\\
\nabla E_t &= {\mathcal P}_{\psi^{\xi_{meta}}} \left( G_t \right),\\
G_t(z) &= \mathbb{E}(\xi(X_t)|Z_t=z),
\end{aligned}
\right.
$$
where $\overline{W}_t$ is a $m-$dimensional Brownian motion independent of $W_t$. The results of Theorem~\ref{theo1} apply to this extended PABF dynamics.
\end{enumerate}

\end{remark}

\section{Numerical experiments}\label{num}

\subsection{Presentation of the model}
 We consider a system composed of $N$ particles $(q_i)_{0\leq i\leq N-1}$ in a two-dimensional periodic box of side length $L$. Among these particles, three particles (numbered $0$, $1$ and $2$ in the following) are designated to form a trimer, while the others are solvent particles. In this model, a trimer is a molecule composed of three identical particles linked together by two bonds (see Figure~\ref{tri}).
 
  \subsubsection{Potential functions}

All particles, except the three particles forming the trimer, interact through the purely repulsive WCA pair potential, which is the Lennard-Jones (LJ) potential truncated at the LJ potential minimum:
$$V_{WCA}(d)=\left\{
\begin{array}{cc}
\varepsilon+4\varepsilon\left[ \left(\frac{\sigma}{d}\right)^{12}-\left(\frac{\sigma}{d}\right)^{6} \right]&\qquad \mbox{if} \, d \leq d_0,\\
0 &  \qquad \mbox{if} \, d \geq d_0 ,\\
\end{array}
\right.
$$
where $d$ denotes the distance between two particles, $\varepsilon$ and $\sigma$ are two positive parameters and $d_0 = 2^{1/6}\sigma$.

A particle of the solvent and a particle of the trimer also interact through the potential $V_{WCA}$. The interaction potential between two particles of the trimer ($q_0/q_1$ or $q_1/q_2$) is a double-well potential (see Figure~\ref{pot2}):
\begin{equation}\label{vs}
V_{S}(d)=h\left[ 1- \frac{(d-d_1-\omega)^2}{\omega^2}\right]^{2},
\end{equation}
where $d_1$, $h$ and $\omega$ are positive parameters.\\
The potential $V_{S}$ has two energy minima. The first one, at $d=d_1$, corresponds to the compact bond. The second one, at $d=d_1+2\omega$, corresponds to the stretched bond. The height of the energy barrier separating the two states is $h$. 

In addition, the interaction potential between $q_0$ and $q_2$ is a Lennard-Jones potential:
 $$V_{LJ}(d)=4\varepsilon '\left[\left(\frac{\sigma'}{d}\right)^{12}-\left(\frac{\sigma'}{d}\right)^{6} \right],$$
where $\varepsilon '$ and $\sigma '$ are two positive parameters.

Finally, the three particles of the trimer also interact through the following potential function on the angle $\theta$ formed by the vectors $\overrightarrow{q_1q_0}$ and $\overrightarrow{q_1q_2}$:
 $$V_{\theta_0}(\theta)=\frac{k_{\theta}}{2}\left(\cos(\theta)-\cos(\theta_0) \right)^2,$$
where $\theta_0$ is the equilibrium angle and $k_{\theta}$ is the angular stiffness. Figure~\ref{tri} presents a schematic view of the system.

   \begin{figure}[htp]
    \centering
  \begin{tabular}{ccc}
    \includegraphics[width=40mm]{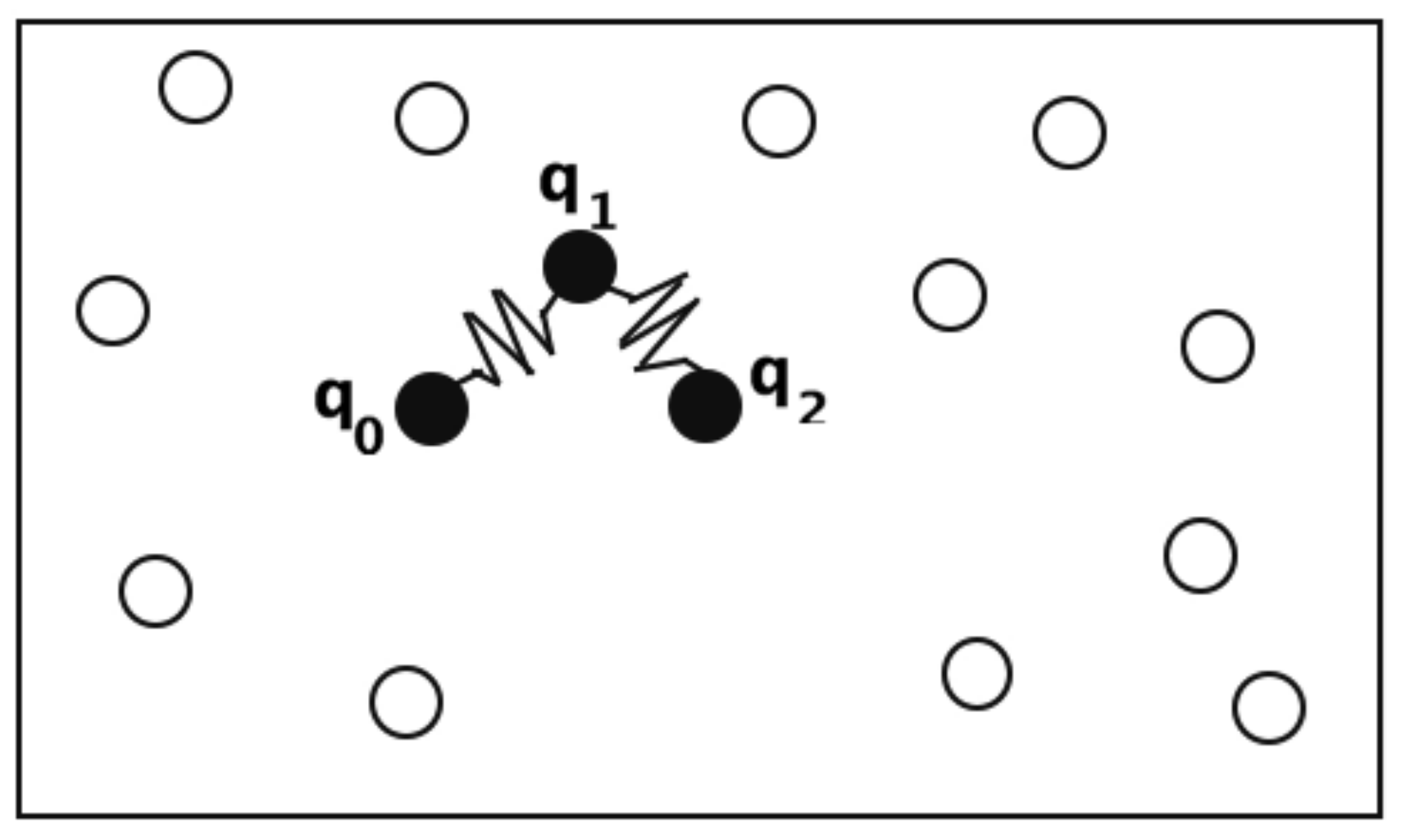}&
    \includegraphics[width=40mm]{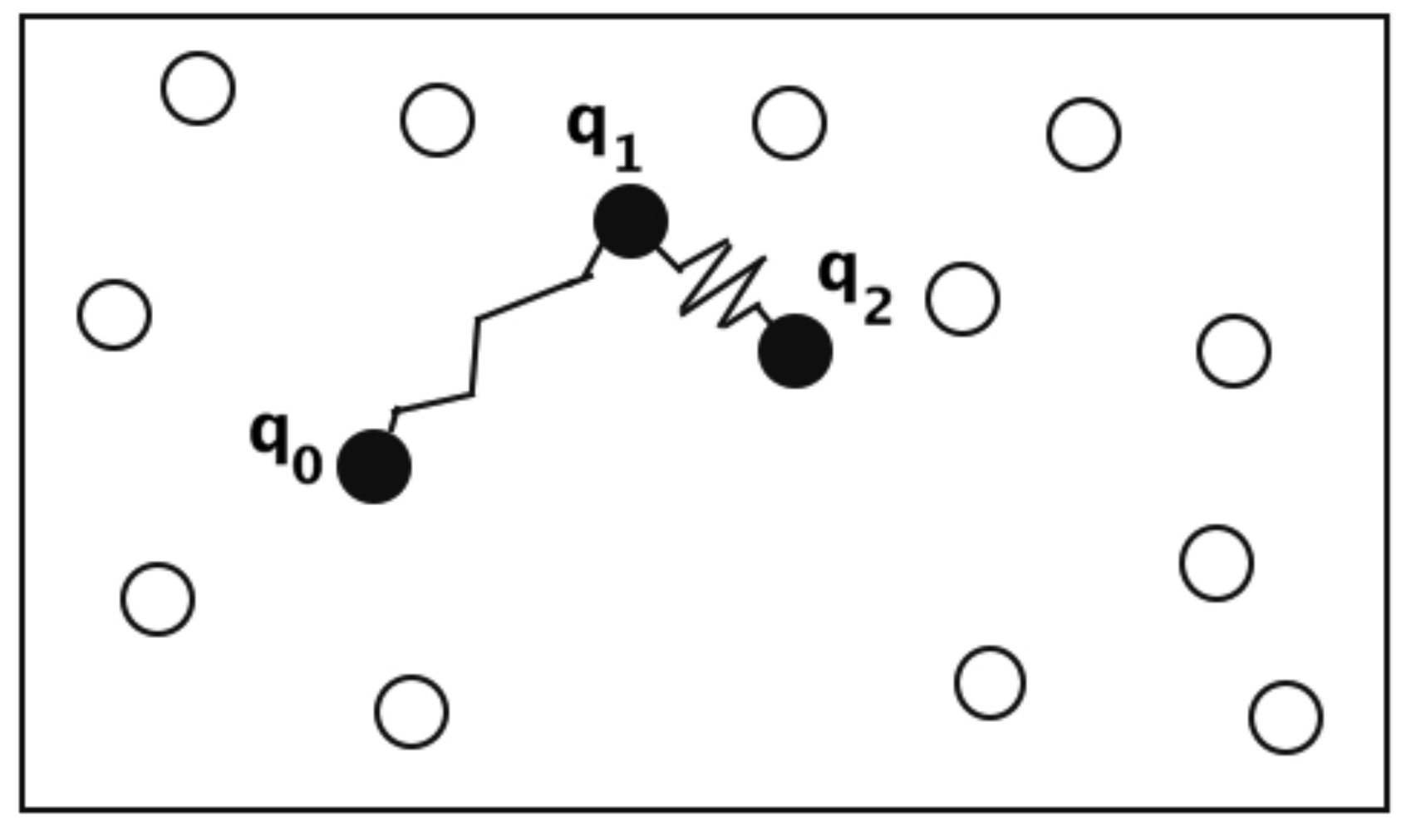}&
    \includegraphics[width=40mm]{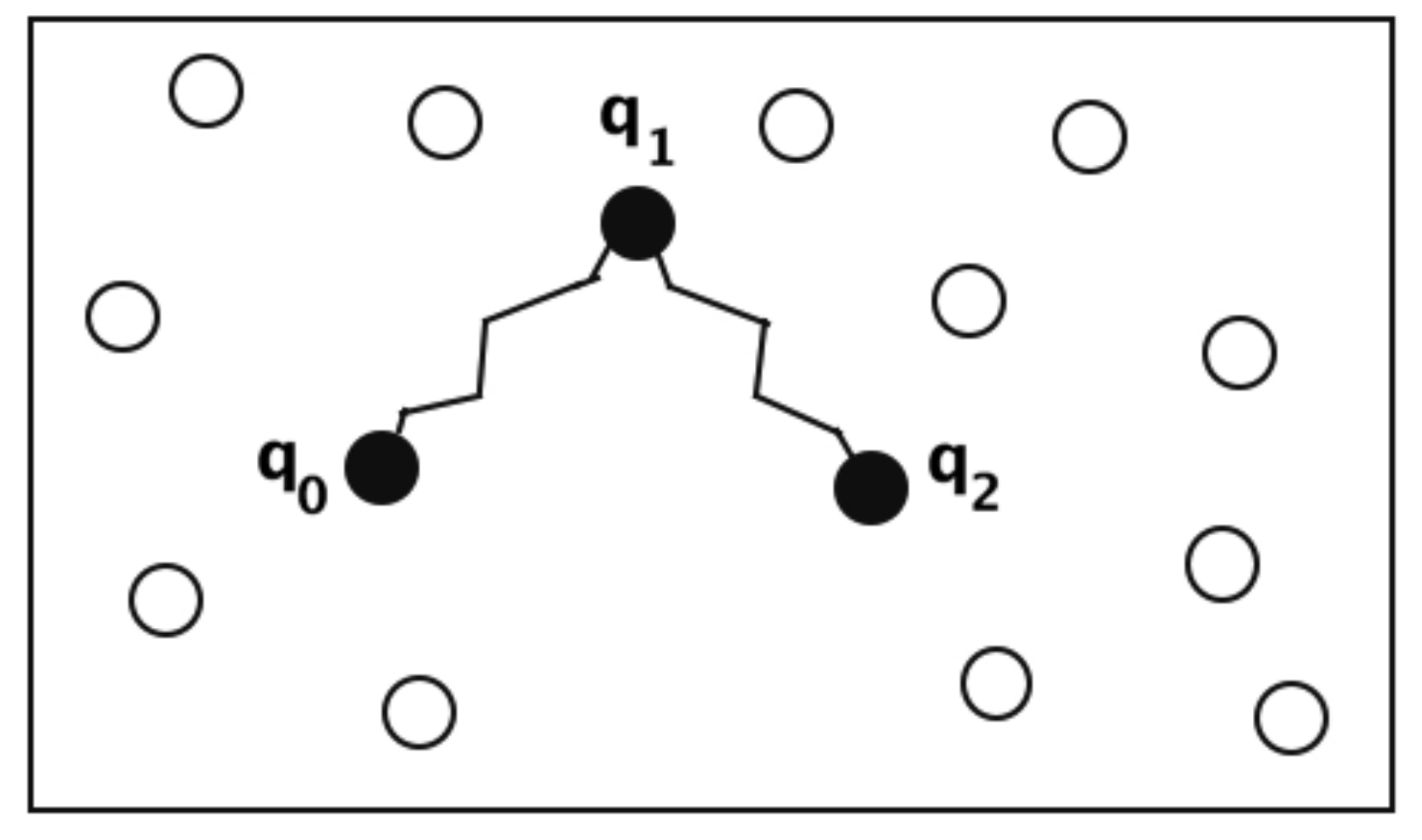}\\
  \end{tabular}
\caption{Trimer ($q_0, q_1, q_2$). Left: compact state; Center: mixed state; Right: stretched state.}\label{tri}
\end{figure}
 
 The total energy of the system is therefore, for $q \in (L\mathbb{T})^{2N}$: 
 \begin{align*}
  V(q)=&\displaystyle \sum_{3\leq i<j\leq N-1}V_{WCA}(|q_i-q_{j}|)+\sum_{i=0}^{2}\sum_{j=3}^{N-1}V_{WCA}(|q_i-q_{j}|)\\
&+  \sum_{i=0}^{1} V_S(|q_i-q_{i+1}|) + V_{LJ}(|q_0-q_{2}|) + V_{\theta_0}(\theta).
 \end{align*}
 
 \begin{figure}[htbp]  
 \centerline{\psfig{figure=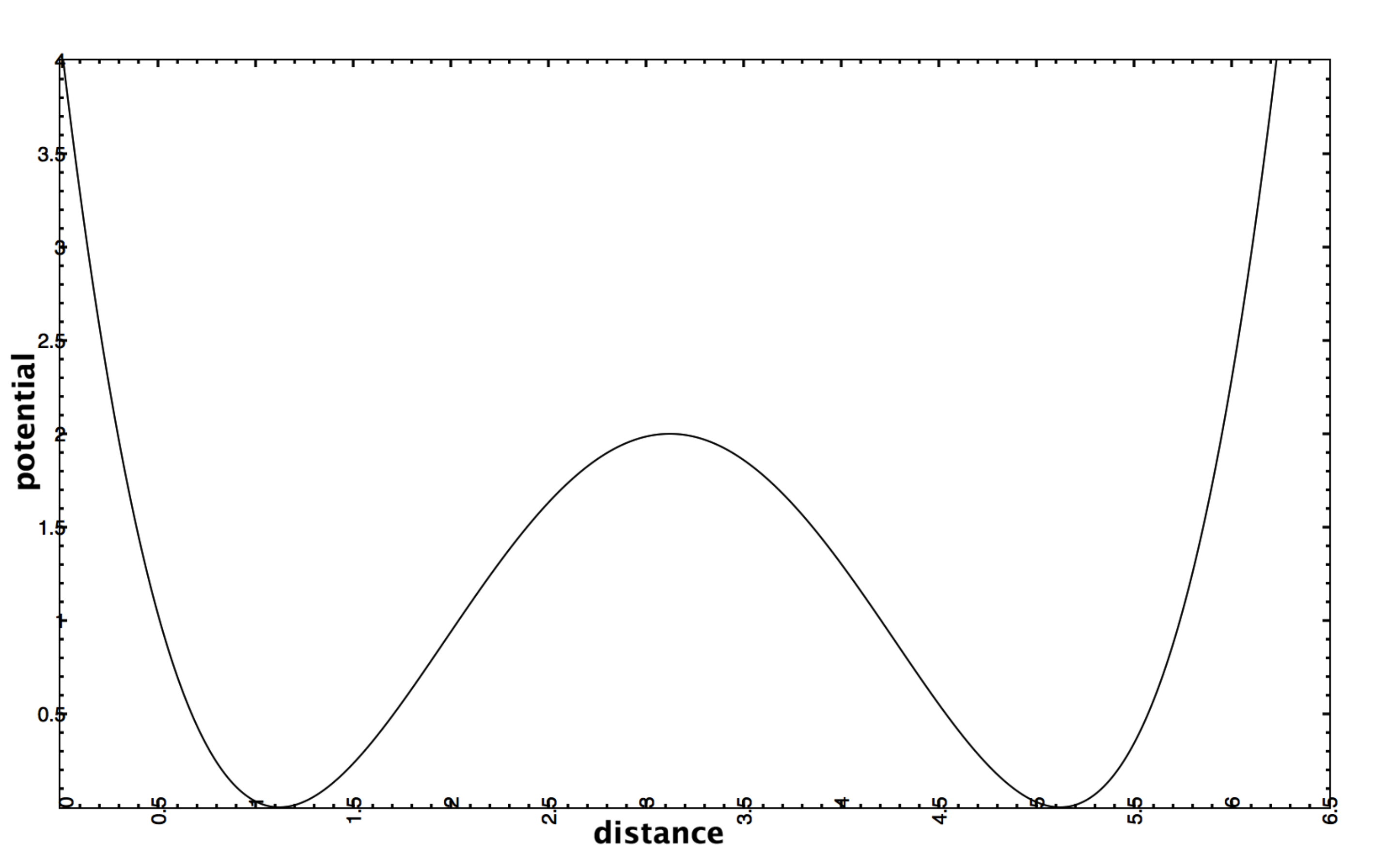,height=50mm}}
\caption{Double-well potential \eqref{vs}, with $d_1=2^{1/6}$, $\omega=2$ and $h=2$.}\label{pot2}
\end{figure}

    \subsubsection{Reaction coordinate and physical parameters}
The reaction coordinate describes the transition from compact to stretched state in each bond. It is the normalised bond length of each bond of the trimer molecule. More precisely, the reaction coordinate is $\xi=(\xi_1,\xi_2)$, with $\xi_1(q)=\frac{|q_0- q_1|-d_0}{2\omega}$ and $\xi_2(q)=\frac{|q_1- q_2|-d_0}{2\omega}$. For $i=1,2$, the value $\xi_i=0$ refers to the compact state (i.e. $d=d_0$) and the value $\xi_i=1$ corresponds to the stretched state (i.e. $d=d_0+2\omega$).
 
 We apply ABF and PABF dynamics to the trimer problem described above. The inverse temperature is $\beta = 1$, we use $N = 100$ particles ($N-3$ solvent particles and the trimer) and the box side length is $L=15$. The parameters describing the WCA and the Lennard-Jones interactions are set to $\sigma = 1$, $\varepsilon=1$, $\sigma '=1$, $\varepsilon '=0.1$, $d_0=2^{1/6}$, $d_1=2^{1/6}$ and the additional parameters for the trimer are $\omega = 2$ and $h = 2$. The parameters describing the angle potential are: $\theta_0$ such that $\cos(\theta_0)=1/3$ and $k_{\theta}=1$ (we refer to [\refcite{rap:97}], Section 10.4.2, for the choice of such parameters). The initial condition on the trimer is as follows: Both bonds $q_0q_1$ and $q_1q_2$ are in compact state, which means that the distance between $q_0$ and $q_1$ and the distance between $q_1$ and $q_2$ are equal to $d_0$. Moreover, the initial bond angle is $\theta_0$.

\subsubsection{Numerical methods and numerical parameters}  
Standard and projected ABF methods are used with $N_{replicas} = 100$ replicas of the system evolving according to the overdamped Langevin dynamics discretized with a time-step $\Delta t = 2.5\times 10^{-4}$. 
 The reaction coordinate space of interest is taken of the form $\mathcal{M}= [\xi_{min}, \xi_{max}]\times  [\xi_{min}, \xi_{max}] $, where $\xi_{min}=-0.2$ and $\xi_{max}=1.2$. $\mathcal{M}$ is discretized into $N_{bins}\times N_{bins} = 50\times 50=2500$ bins of equal sizes and $\delta=\delta_x=\delta_y=\frac{\xi_{max}-\xi_{min}}{N_{bins}}=0.028$ denotes the size of each bin along both axes. 
 
To implement the ABF and PABF method, one needs to approximate $F_t^i(x,y)=\displaystyle\mathbb{E}[f_i(x,y)|\xi(x,y)=(\xi_1(x,y),\xi_2(x,y))],\,i=1,2$. The mean force $F_t$ is estimated in each bin as a combination of plain trajectorial averages and averages over replicas. It is calculated at each time as an average of the local mean force in the bin over the total number of visits in this bin. More precisely, at time $t$ and for $l=1,2$, the value of the mean force in the $(i,j)^{th}$ bin is: 
 \begin{equation}\label{F_t}
  F_t^{l}(i,j)=\displaystyle \frac{\displaystyle\sum_{t'\leq t}\sum_{k=1}^{N_{replicas}}f_{l}(q_{k,t'})1_{\{indx(\xi(q_{k,t'}))=(i,j)\}}}{\displaystyle\sum_{t'\leq t}\sum_{k=1}^{N_{replicas}}1_{\{indx(\xi(q_{k,t'}))=(i,j)\}}},
 \end{equation}
where $q_{k,t}$ denotes the position $(x_k,y_k)$ at time $t$, $f=(f_1,f_2)$ is defined in \eqref{cmf} and $indx(\xi(q_{k,t'}))$ denotes the number of the bin where $\xi(q_{k,t'})$ lives, i.e.
 $$indx(\xi(q))=\displaystyle\left(\left[\frac{\xi_1(q)-\xi_{min}}{\delta}\right]_+,\left[\frac{\xi_2(q)-\xi_{min}}{\delta}\right]_+\right),\,\forall q\in\mathcal{M}.$$ 
If the the components of the index function (i.e. $indx$) are either equal to $-1$ or to $N_{bins}$, it means that we are outside $\mathcal{M}$ and then the confining potential is non zero, and defined as:
 $$W(\xi(q_{k,t}))=\displaystyle\sum_{i=1}^2\left[ 1_{\{\xi_i(q_{k,t})\geq \xi_{max}\}}(\xi_i(q_{k,t})-\xi_{max})^2+1_{\{\xi_i(q_{k,t})\leq \xi_{min}\}}(\xi_i(q_{k,t})-\xi_{min})^2\right].$$
To construct the PABF method, the solution to the following Poisson problem with Neumann boundary conditions are approximated:
  \begin{equation}\label{poistri}
\left\{
\begin{array}{llll}
\Delta A&=&{\rm div}F&\quad\mbox{in }\mathcal{M}= [\xi_{min}, \xi_{max}]\times  [\xi_{min}, \xi_{max}],\\
\frac{\partial A}{\partial n}&=&F\cdot n&\quad \mbox{on }\partial\mathcal{M},
\end{array}
\right.
\end{equation}	
where $n$ denotes the unit normal outward to $\mathcal{M}$. We use for simplicity in the numerical experiments the standard Helmholtz problem, without the weight $\psi^{\xi}$. Problem~\eqref{poistri} is solved using finite element method of type $Q^1$ on the quadrilateral mesh defined above, with nodes $(x_i,y_j)$, where $x_i=\xi_{min}+i\delta$ and $y_j=\xi_{min}+j\delta$, for $i,j=0,.., N_{bins}$. The space $\mathcal{M}$ is thus discretized into $N_T=N_{bins}^2$ squares, with $N_s=(N_{bins}+1)^2$ nodes.  The associated variational formulation is the following:\\
   \begin{equation*}
\left\{
\begin{array}{l}
 \mbox{Find}\, A\in H^1(\mathcal{M})/\mathbb{R}\mbox{ such that}\\
 \displaystyle \int_{\mathcal{M}}\nabla A\cdot\nabla v=\int_{\mathcal{M}}F\cdot\nabla v,\,\forall\, v\in  H^1(\mathcal{M})/\mathbb{R}.
\end{array}
\right.
\end{equation*}

\subsection{ Comparison of the methods}
  In this section, we compare results obtained with three different simulations: without ABF, with ABF and with projected ABF (PABF). First, it is observed numerically that both ABF methods overcome metastable states. Second, it is illustrated how PABF method reduces the variance of the estimated mean force compared to ABF method. As a consequence of this variance reduction, we observe that the convergence of $\nabla A_t$ to $\nabla A$ with the PABF method is faster than the convergence of $F_t$ to $\nabla A$ with the ABF method. 
  
  \subsubsection{Metastability}
To illustrate numerically the fact that ABF methods improve the sampling for metastable processes, we observe the variation, as a function of time, of the two metastable distances (i.e. the distance between $q_0$ and $q_1$, and the distance between $q_1$ and $q_2$). On Figures~\ref{dst1} and \ref{dst2}, the distance between $q_1$ and $q_2$ is plotted as a function of time for three dynamics: without ABF, ABF and PABF.

Both ABF methods allow to switch faster between the compact and stretched bond and thus to better explore the set of configurations. Without adding the biasing term, the system remains trapped in a neighborhood of the first potential minimum (i.e. $d_0\simeq 1.12$) region for $20$ units of time at least (see Figure~\ref{dst1}), while when the biasing term is added in the dynamics, many jumps between the two local minima are observed (see Figure~\ref{dst2}). 

 \begin{figure}[htp]
    \centering
    \includegraphics[width=62mm]{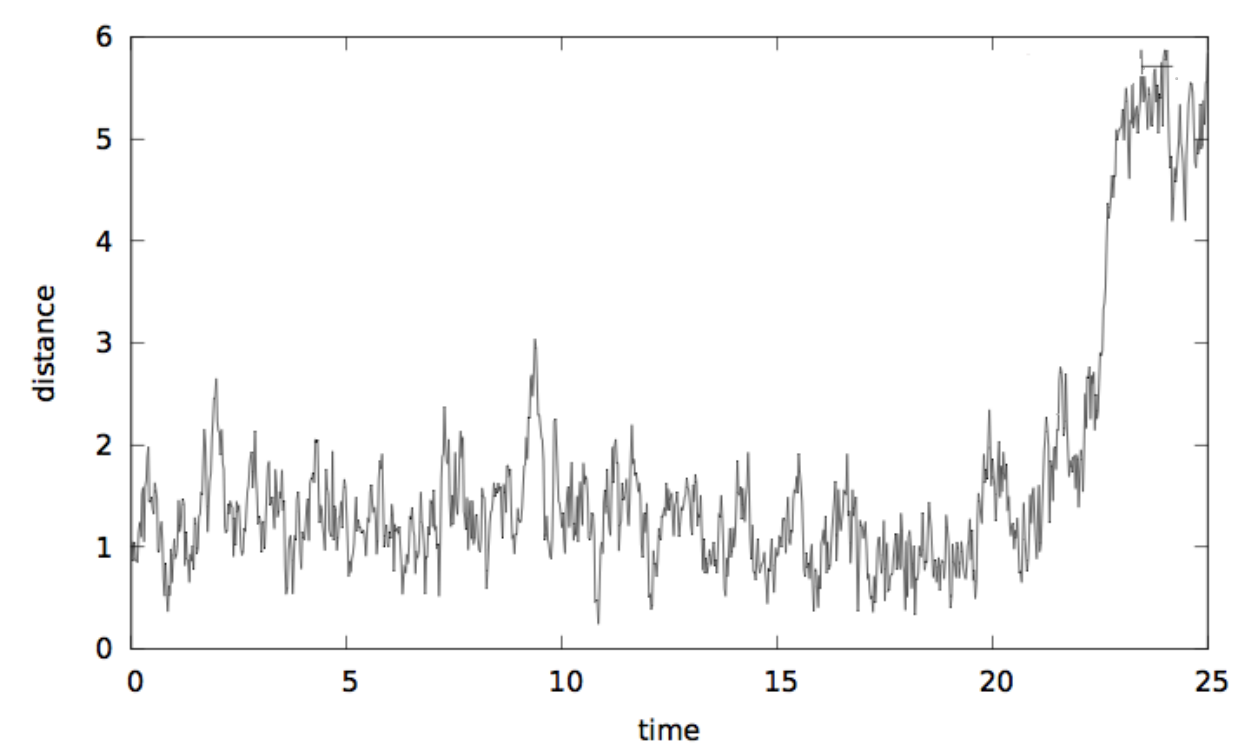}\\
\caption{Without ABF.}\label{dst1}
\end{figure}

 \begin{figure}[htp]
    \centering
  \begin{tabular}{ccc}
    \includegraphics[width=62mm]{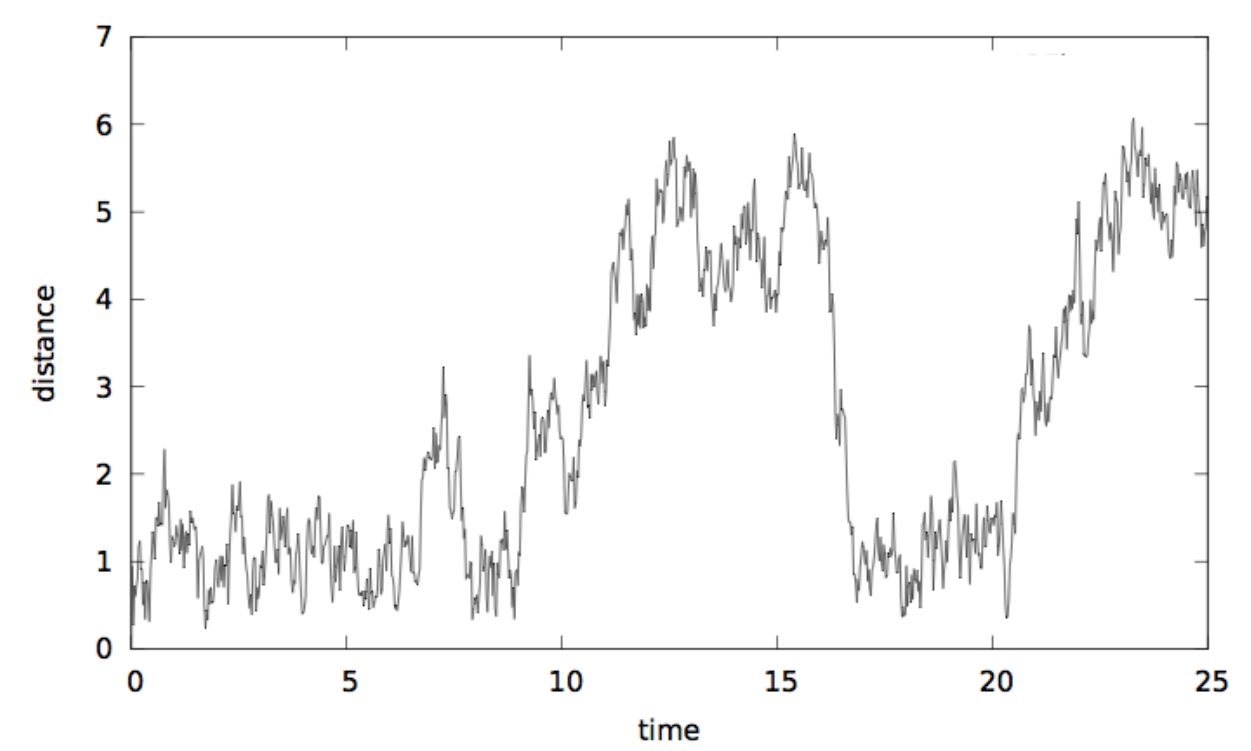}&
    \includegraphics[width=62mm]{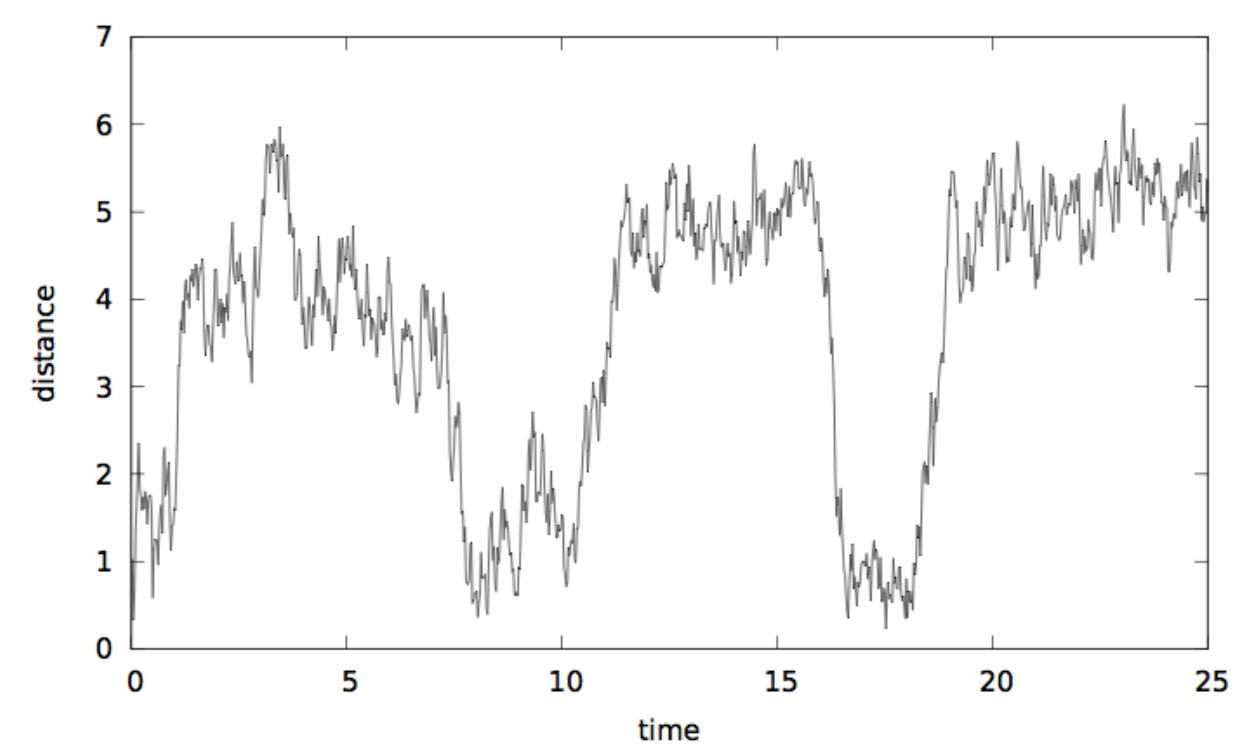}
  \end{tabular}
\caption{Left: ABF; Right: PABF.}\label{dst2}
\end{figure}

 \subsubsection{Variance reduction}
 
Since we use Monte-Carlo methods to approximate $\nabla A_t$, the variance is an important quantity to assess the quality of the result. The following general proposition shows that projection reduces the variance.
\begin{proposition}\label{vared}
Let $F$ be a random function from $ \mathbb{T}^2$ into $ \mathbb{R}^2$ and belongs to $H(\mathrm{div},\mathbb{T}^2)$, and define $\mathcal{P}=\mathcal{P}_1$ (i.e. without weight) the projection on gradient vector fields defined in Section \ref{minpb}. Then, the variance of $\mathcal{P}(F)$ is smaller than the variance of $F$:
$$\displaystyle \int_ \mathcal{M}{\rm Var}(\mathcal{P} (F))\leq \int_ \mathcal{M}{\rm Var}(F),$$
where, for any vector field $F$, $\mathrm{Var}(F)=\mathbb{E}(|F|^2)-\mathbb{E}(|F|)^2$ and $|F|$ being the Euclidian norm.
\end{proposition}
\begin{proof}
Let $F$ be a random vector field of $ H({\rm div},{\mathbb{T}^2})$. Let us introduce $\mathcal{P}(F)\in H^1(\mathbb{T}^2)\times H^1(\mathbb{T}^2)$ its projection. Notice that by the linearity of the projection $\mathcal{P}(\mathbb{E}(F))=\mathbb{E}(\mathcal{P}(F))$. By definition of $\mathcal{P}(F)$, one gets:
$$\displaystyle\int_{\mathbb{T}^2}(F-\mathcal{P}(F))\cdot\nabla h=0,\,\forall h\in H^1(\mathbb{T}^2).$$
Therefore, using Pythagoras and the fact that $\mathcal{P}(F)$ is agradient,
$$\displaystyle\int_{\mathbb{T}^2}|F|^2=\int_{\mathbb{T}^2}|F-\mathcal{P}(F)|^2+\int_{\mathbb{T}^2}|\mathcal{P}(F)|^2$$
and
$$\displaystyle\int_{\mathbb{T}^2}|F-\mathbb{E}(F)|^2=\int_{\mathbb{T}^2}|F-\mathbb{E}(F)-\mathcal{P}(F-\mathbb{E}(F))|^2+\int_{\mathbb{T}^2}|\mathcal{P}(F-\mathbb{E}(F))|^2.$$
Using the linearity of $\mathcal{P}$, we thus obtain
$$\displaystyle\int_{\mathbb{T}^2}{\rm Var}(F)=\int_{\mathbb{T}^2}{\rm Var}(F-\mathcal{P}(F))+\int_{\mathbb{T}^2}{\rm Var}(\mathcal{P}(F)),$$
which concludes the proof.

\end{proof}
We illustrate the improvement of the projected method in terms of the variances of the biasing forces, by comparing ${\rm Var}(\nabla A_t)={\rm Var}(\partial_1 A_t)+{\rm Var}(\partial_2 A_t)$ (for the PABF method) and ${\rm Var}(F_t)={\rm Var}(F_t^1)+{\rm Var}(F_t^2)$ (for the ABF method). Figure~\ref{var} shows that the variance for the projected ABF method is smaller than for the standard ABF method.

We have $N_{bins}\times N_{bins}=2500$ variable for each term (i.e. $\partial_1A_t$, $\partial_2 A_t$, $F_t^1$ and $F_t^2$). The variances are computed using $20$ independent realizations as follows:
$$\displaystyle {\rm Var}( F_t^1)=\frac{1}{2500}\displaystyle \sum_{i,j=1}^{50}\frac{1}{20}\sum_{k=1}^{20}F_t^{1,k}(x_i,y_j)^2 - \frac{1}{2500}\displaystyle \sum_{i,j=1}^{50}\left (\frac{1}{20}\sum_{k=1}^{20}F_t^{1,k}(x_i,y_j)\right )^2.$$
 Note that four averages are involved in this formula: an average with respect to the space variable, an average over the  $20$ Monte-Carlo realizations, an average over replicas and a trajectorial average (the last two averages are more explicit in \eqref{F_t}). Notice that since the variance of the biasing force is smaller with PABF, one may expect better convergence in time results. This will be investigated in Section~\ref{error} and Section~\ref{distrint}.
 
 \begin{figure}[htp]
  \centering
    \includegraphics[width=100mm]{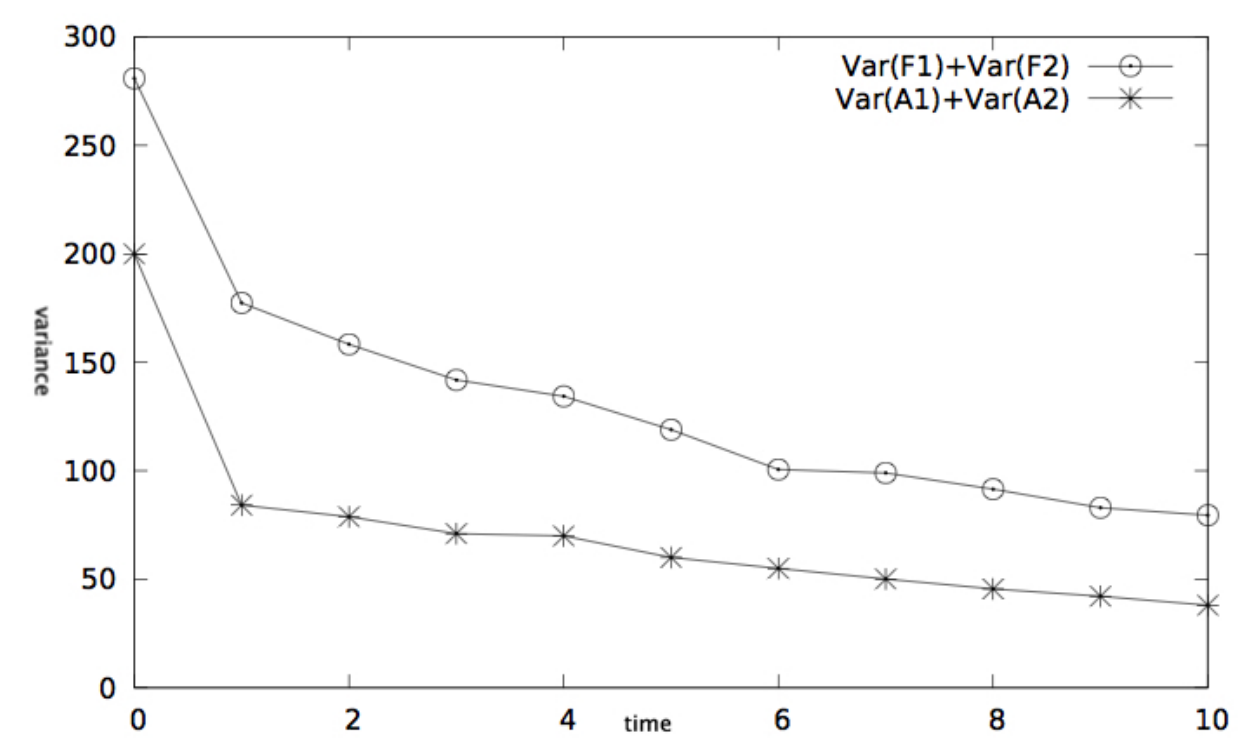}
\caption{Variances as a function of time.}\label{var}
\end{figure}

  \subsubsection{Free energy error}\label{error}

We now present, the variation, as a function of time, of the normalized $L^2-$ distance between the real free energy and the estimated one, in both cases: ABF and PABF methods. As can be seen in Figure~\ref{err}, in both methods, the error decreases as time increases. Moreover, this error is always smaller for the projected ABF method than for the ABF method.
\begin{figure}[htp]
  \centering
    \includegraphics[width=100mm]{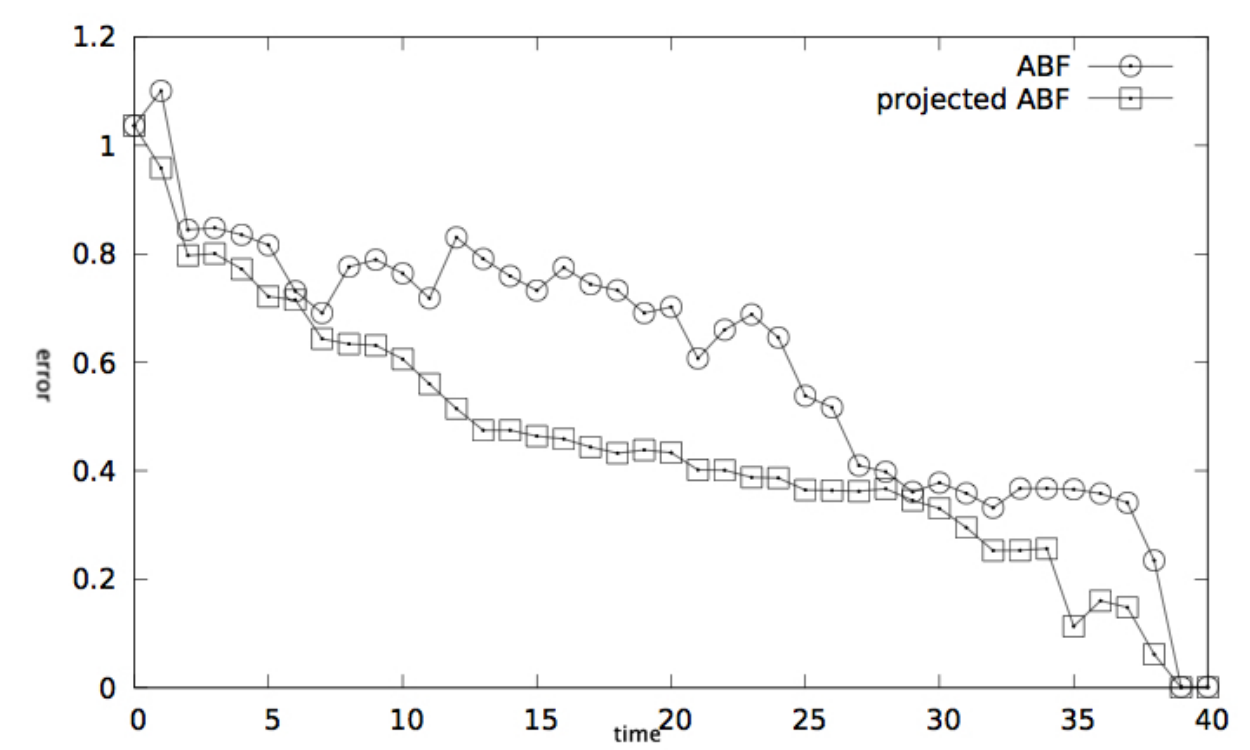}
\caption{ Free energy error as a function of time.}\label{err}
\end{figure}
  \subsubsection{Distribution}\label{distrint}
  
Another way to illustrate that the projected ABF method converges faster than the standard ABF method is to plot the density function $\psi^{\xi}$ as a function of time (see Figure~\ref{dtr0}-\ref{dtr25}). It is illustrated that, as time increases, the probability of visiting all bins (of the reaction coordinate space $\mathcal{M}$) increases. 

It is observed that, for the projected ABF method, the state where both bonds are stretched is visited earlier (at time $5$) than for the standard ABF method (at time $20$). The convergence to uniform law along $(\xi_1,\xi_2)$ is faster with the projected ABF method.

\begin{figure}[htp]
  \centering
  \begin{tabular}{cc}
    \includegraphics[width=40mm, angle=270 ]{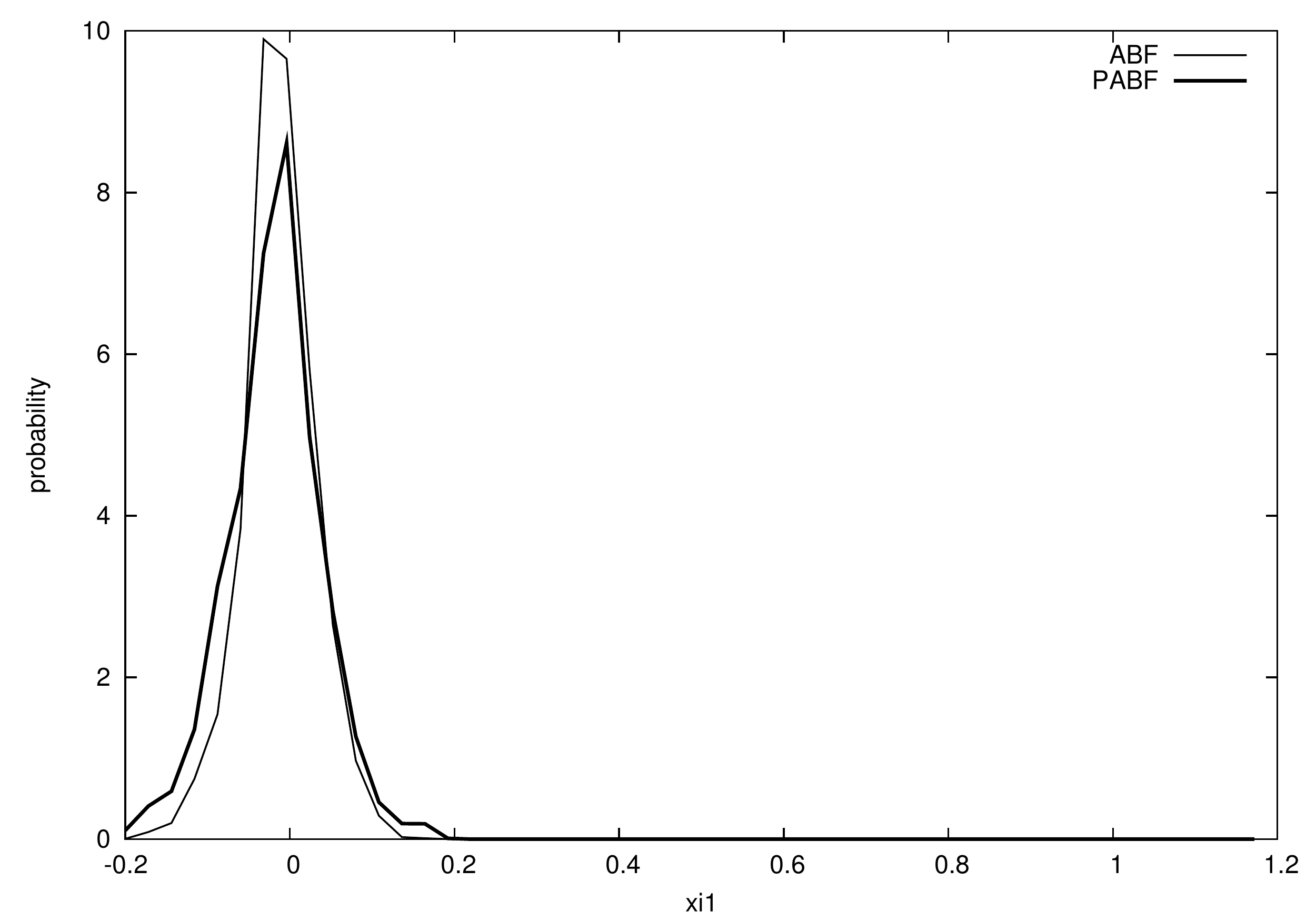}&
    \includegraphics[width=40mm, angle=270 ]{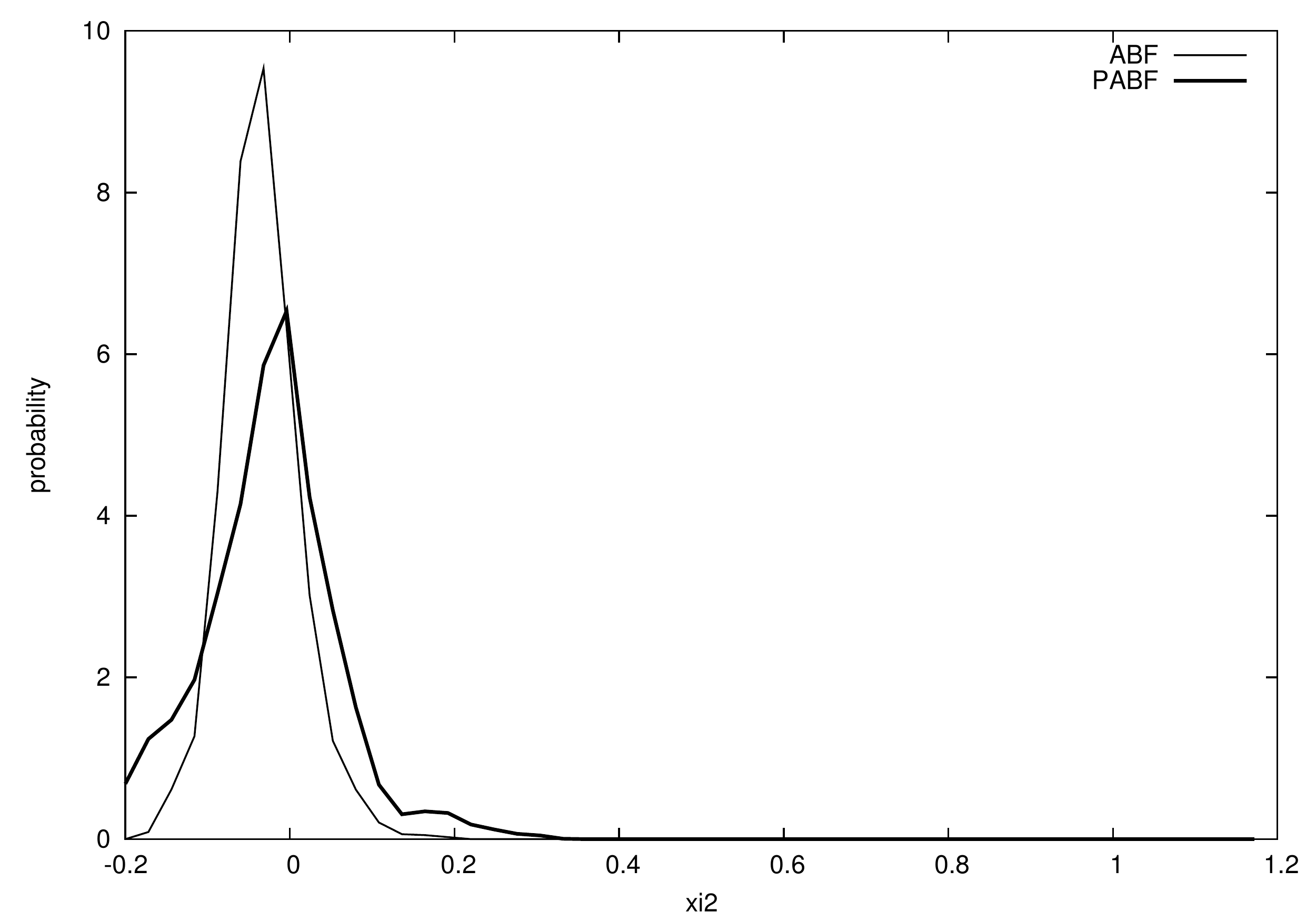}\\
  \end{tabular}
\caption{At time $0.025$. Left: $\int \psi^{\xi}(x_1,x_2)dx_2$; Right:  $\int \psi^{\xi}(x_1,x_2)dx_1$.}\label{dtr0}
\end{figure}
\begin{figure}[htp]
  \centering
  \begin{tabular}{cc}
    \includegraphics[width=40mm, angle=270]{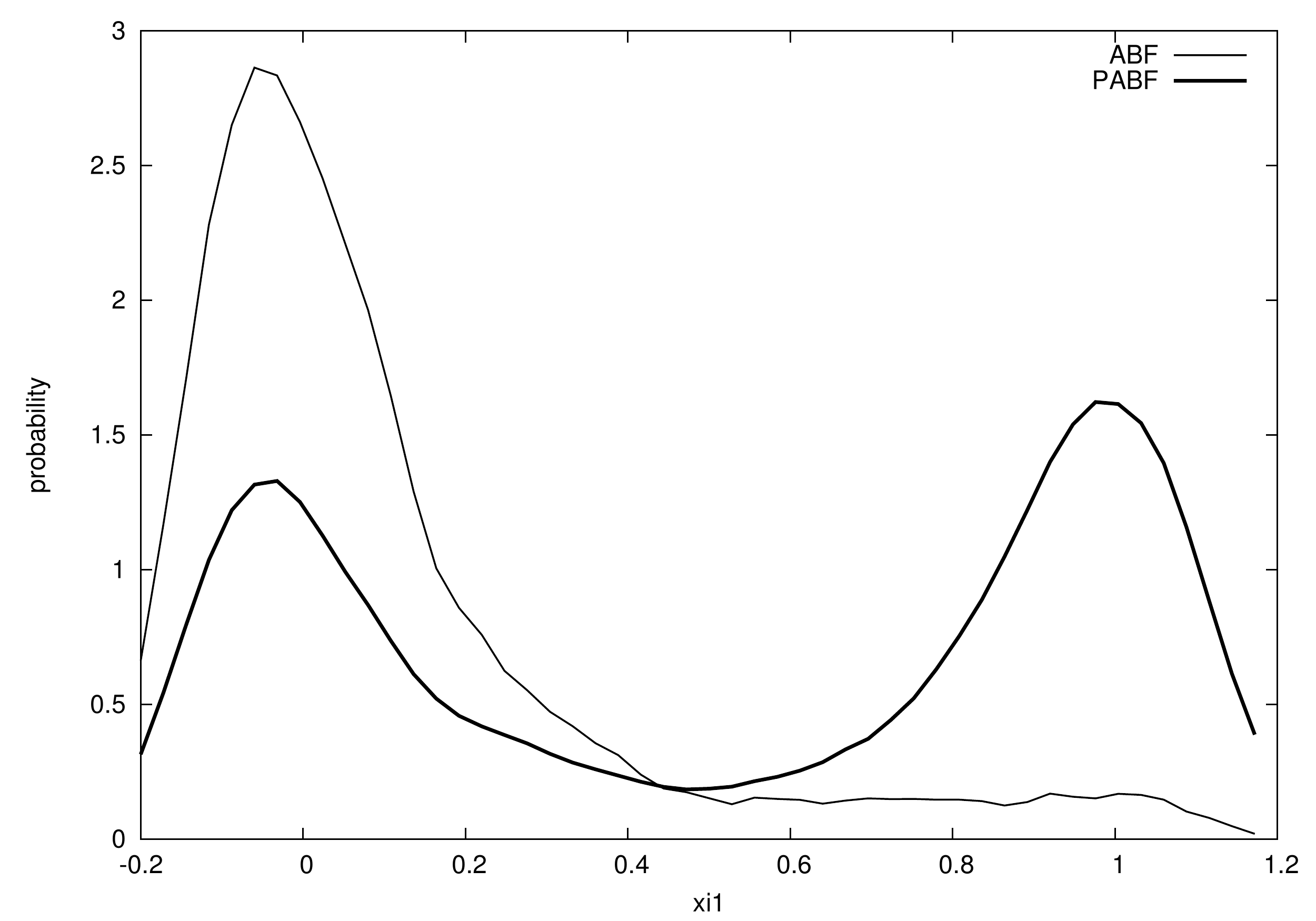}&
    \includegraphics[width=40mm, angle=270]{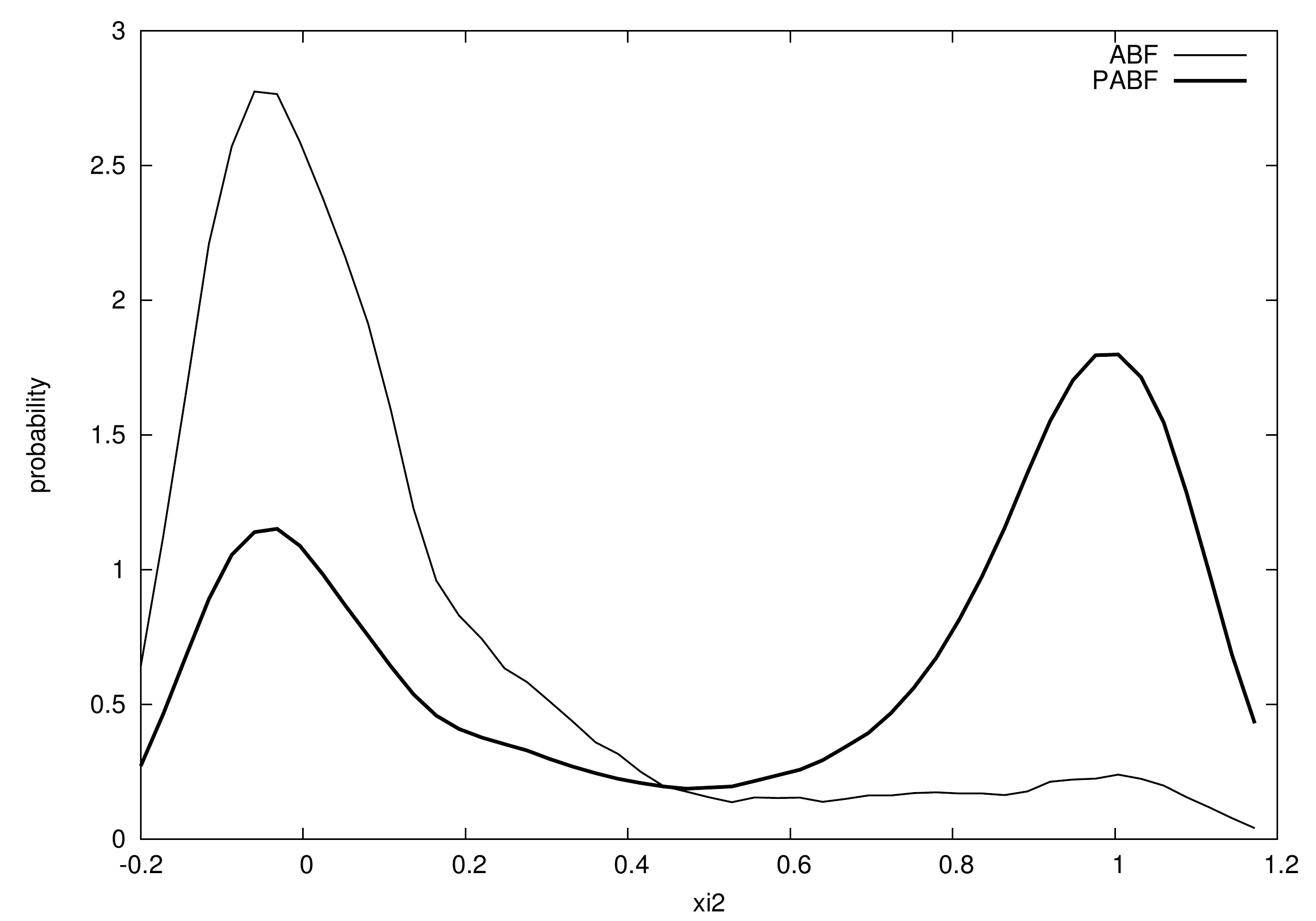}\\
  \end{tabular}
\caption{At time $5$. Left: $\int \psi^{\xi}(x_1,x_2)dx_2$; Right:  $\int \psi^{\xi}(x_1,x_2)dx_1$.}\label{dtr5}
\end{figure}
\begin{figure}[htp]
  \centering
  \begin{tabular}{cc}
    \includegraphics[width=40mm, angle=270]{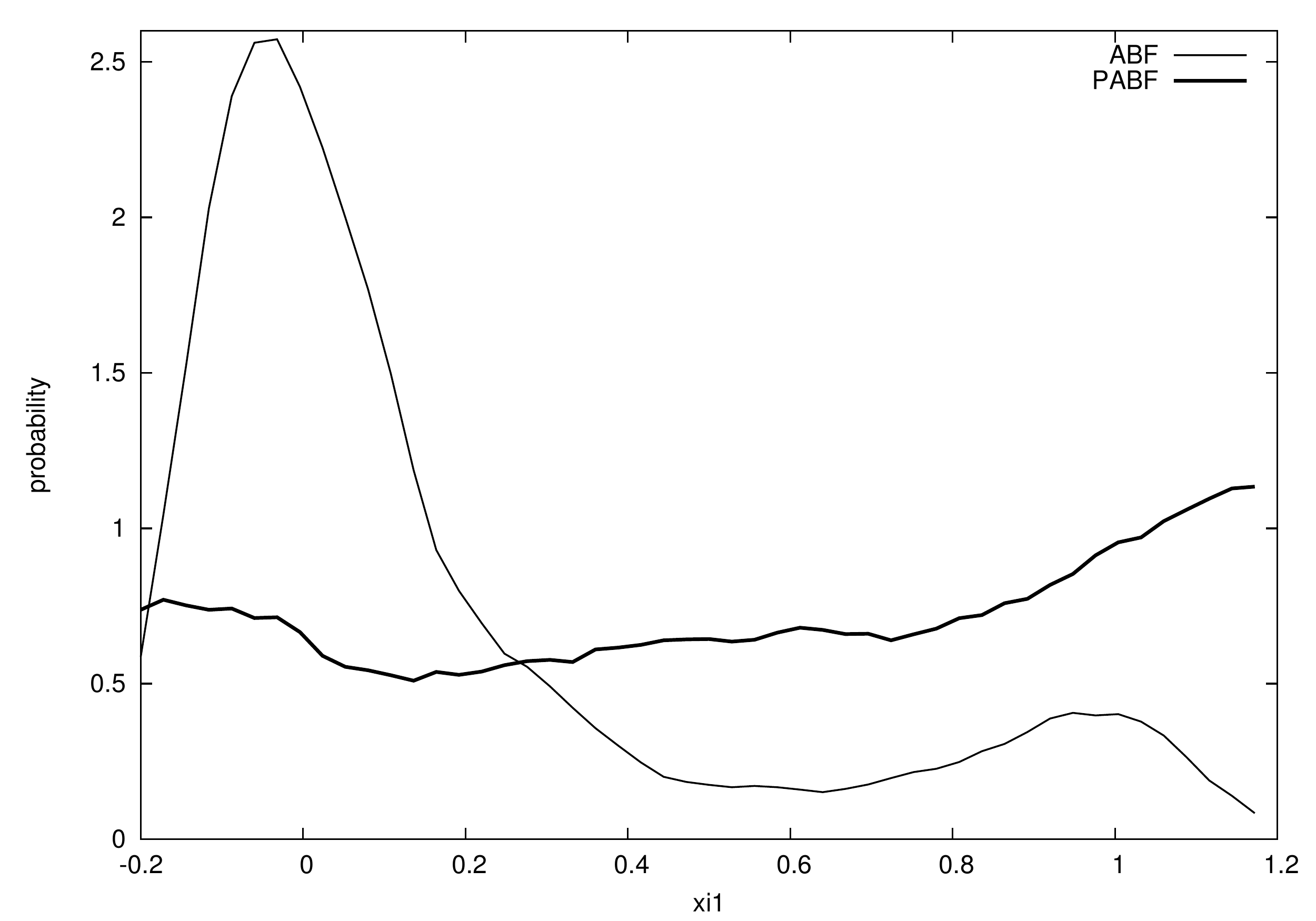}&
    \includegraphics[width=40mm, angle=270]{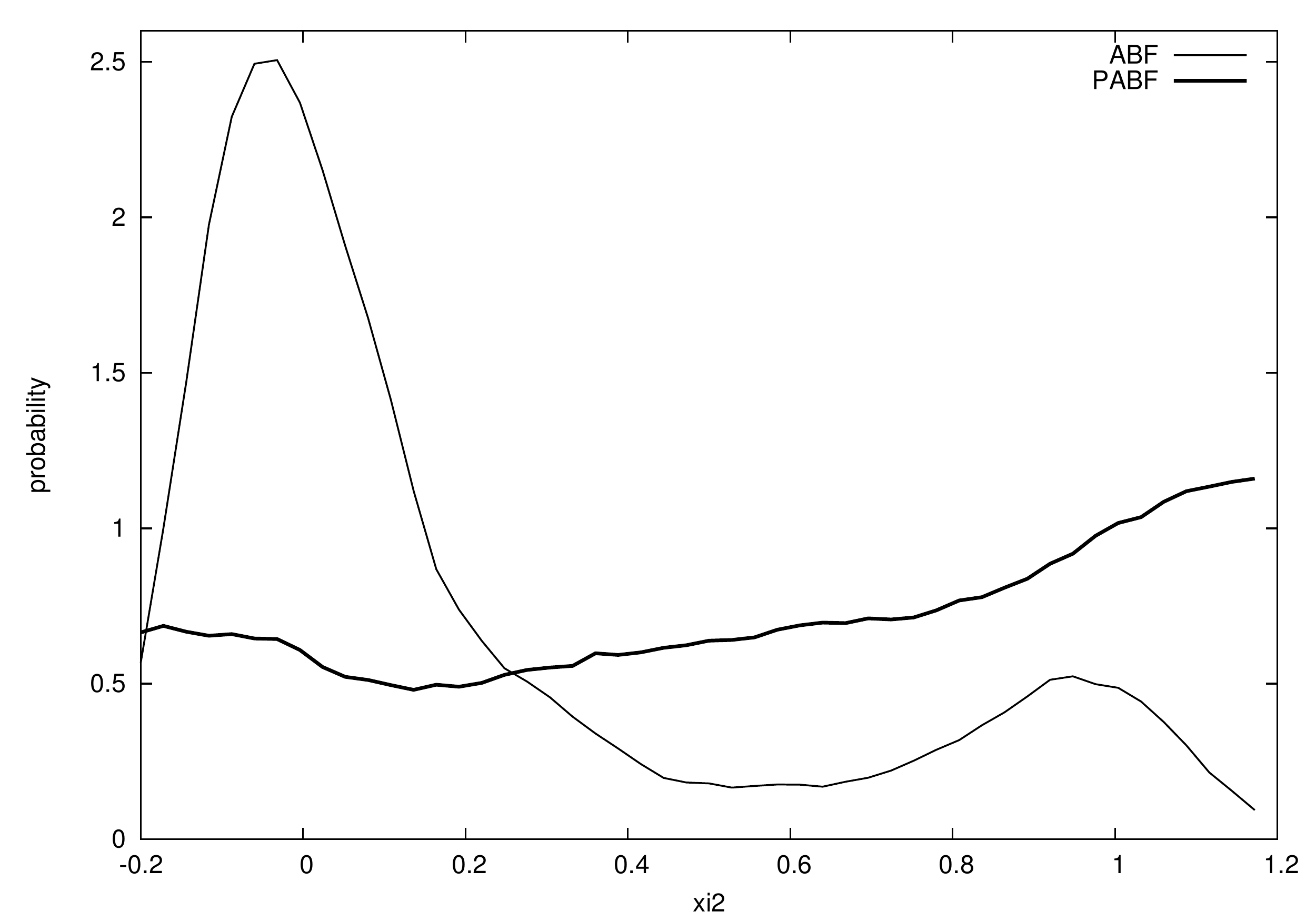}\\
  \end{tabular}
\caption{At time $10$. Left: $\int \psi^{\xi}(x_1,x_2)dx_2$; Right:  $\int \psi^{\xi}(x_1,x_2)dx_1$.}\label{dtr10}
\end{figure}
\begin{figure}[htp]
  \centering
  \begin{tabular}{cc}
    \includegraphics[width=40mm, angle=270]{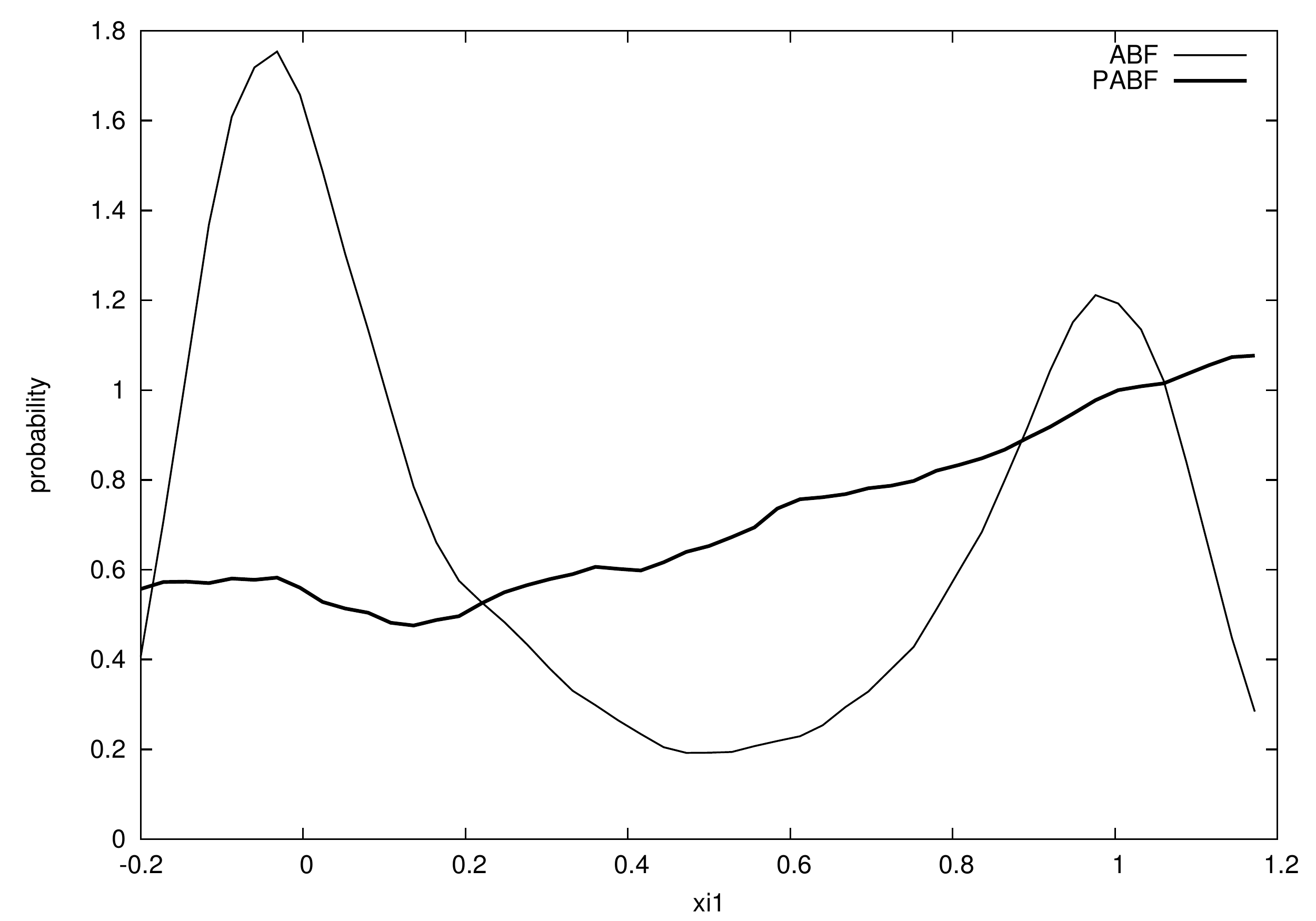}&
    \includegraphics[width=40mm, angle=270]{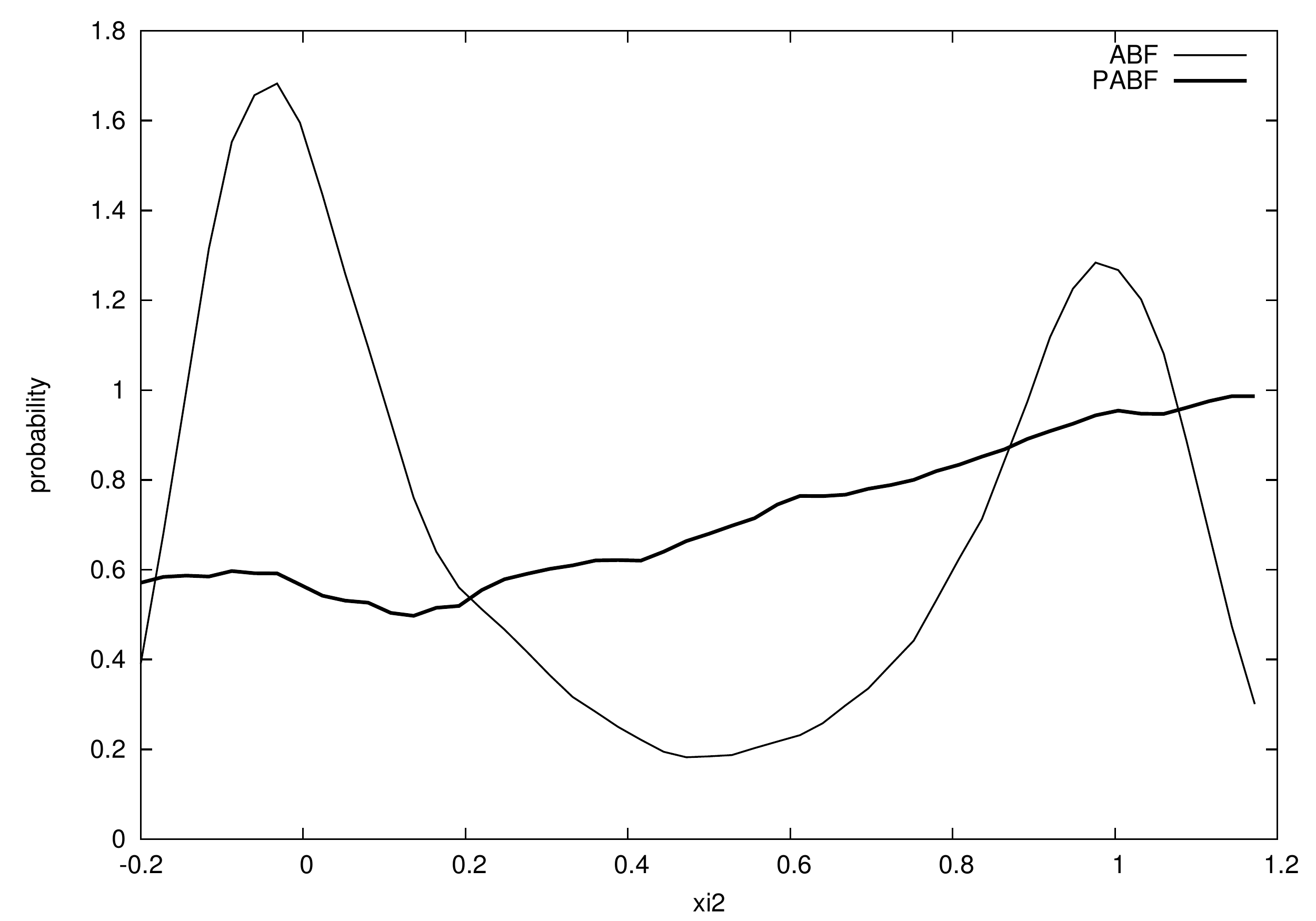}\\
  \end{tabular}
\caption{At time $20$. Left: $\int \psi^{\xi}(x_1,x_2)dx_2$; Right:  $\int \psi^{\xi}(x_1,x_2)dx_1$.}\label{dtr20}
\end{figure}
\newpage
\begin{figure}[htp]
  \centering
  \begin{tabular}{cc}
    \includegraphics[width=40mm, angle=270]{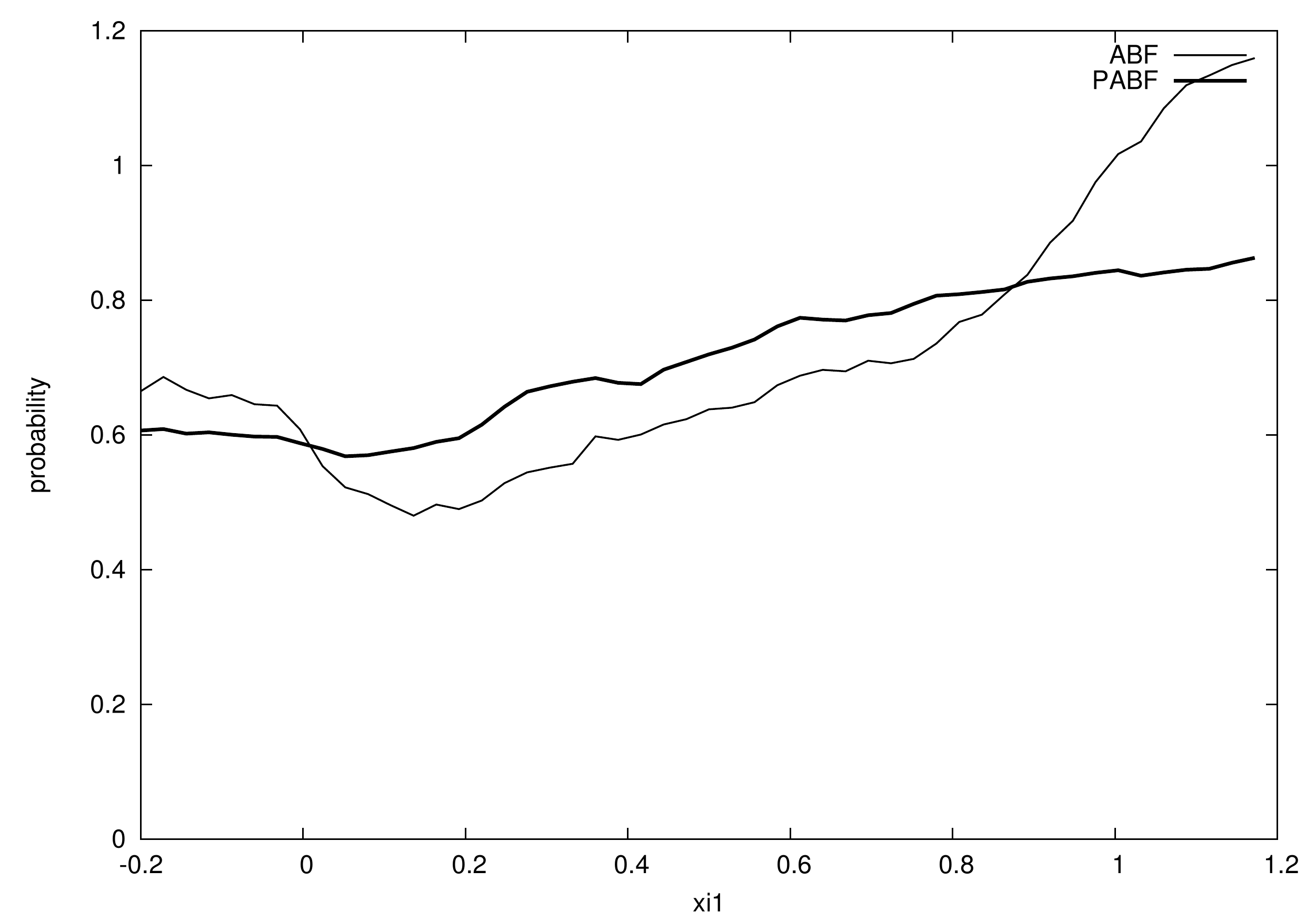}&
    \includegraphics[width=40mm, angle=270]{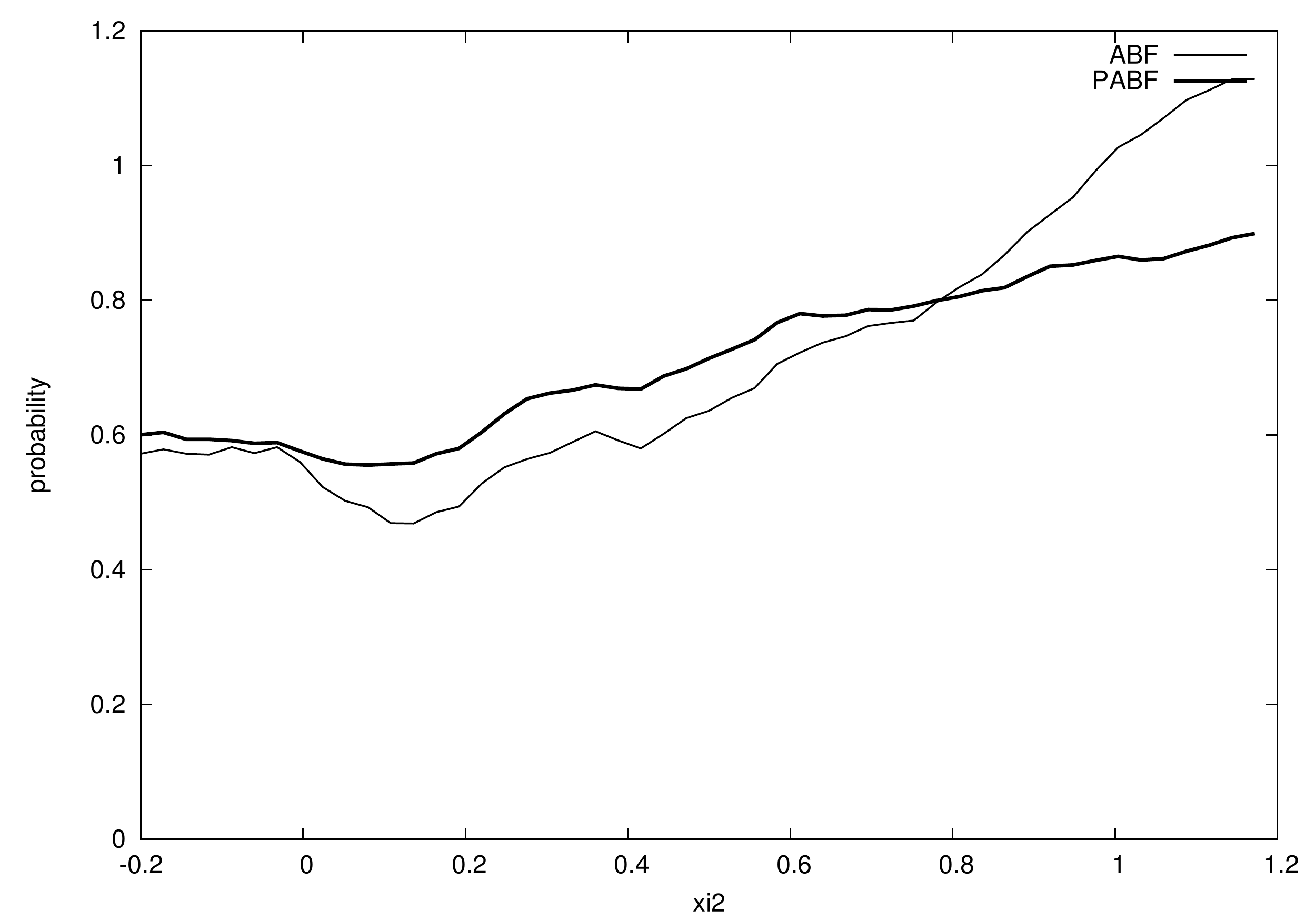}\\
  \end{tabular}
\caption{At time $25$. Left: $\int \psi^{\xi}(x_1,x_2)dx_2$; Right:  $\int \psi^{\xi}(x_1,x_2)dx_1$.}\label{dtr25}
\end{figure}

\section{Proofs}\label{pro}
The proofs are inspired from [\refcite{tony:08}]. One may assume that $\beta=1$ up to the following change of variable: $\tilde{t}=\beta^{-1}t$, $\tilde{\psi}(\tilde{t},x)=\psi(t,x),$ $\tilde{V}(x)=\beta V(x)$. Recall, that we work in $\mathcal{D}=\mathbb{T}^n$, $\mathcal{M}=\mathbb{T}^2$ and $\forall x=(x_1,...,x_n)\in~\mathbb{T}^n$, $\xi(x)=(x_1,x_2)$.

\subsection{\textit{Proof of Proposition \ref{diff}}}\label{difprf}
Let $g:\mathbb{T}^2\rightarrow\mathbb{R}$ be a function in $H^1(\mathbb{T}^2).$
\begin{align*}
\frac{d}{dt}\displaystyle\int_{\mathbb{T}^{2}}\psi^{\xi}gdx_1dx_2&=\frac{d}{dt}\displaystyle\int_{\mathbb{T}^n}\psi g\circ\xi dx_1^n\\
&=\displaystyle\int_{\mathbb{T}^n}{\rm div}[(\nabla V-\sum_{i=1}^2\partial_iA_t\circ\xi\nabla\xi_i)\psi+\nabla\psi]g\circ\xi dx_1^n\\
&=-\displaystyle\int_{\mathbb{T}^n}\sum_{j=1}^2[(\nabla V-\sum_{i=1}^2\partial_iA_t\circ\xi\nabla\xi_i)\psi+\nabla\psi].\nabla\xi_j\partial_jg\circ\xi dx_1^n\\
&=-\sum_{i=1}^2\int_{\mathbb{T}^n}[(\nabla V.\nabla\xi_i\psi+\nabla\psi.\nabla\xi_i]\partial_ig\circ\xi dx_1^n\\
&\quad+\sum_{i=1}^2\int_{\mathbb{T}^n}\partial_iA_t\circ\xi\psi\partial_ig\circ\xi dx_1^n.
\end{align*}
Applying Fubini, it holds:
\begin{align*}
\frac{d}{dt}\displaystyle\int_{\mathbb{T}^{2}}\psi^{\xi}gdx_1dx_2&=-\displaystyle\sum_{i=1}^2\int_{\mathbb{T}^{2}}\int_{\Sigma_{(x_1,x_2)}}[\partial_i V\psi+\partial_i\psi]dx_3^n\partial_ig(x_1,x_2) dx_1dx_2\\
&\quad +\displaystyle\sum_{i=1}^2\int_{\mathbb{T}^{2}}\int_{\Sigma_{(x_1,x_2)}}\partial_iA_t(x_1,x_2)\psi dx_3^n\partial_ig(x_1,x_2) dx_1dx_2\\
&=-\displaystyle\sum_{i=1}^2\int_{\mathbb{T}^{2}} F_t^i\psi^{\xi}\partial_ig(x_1,x_2) dx_1dx_2-\displaystyle\sum_{i=1}^2\int_{\mathbb{T}^{2}}\partial_i\psi^{\xi}\partial_ig(x_1,x_2) dx_1dx_2\\
&\quad +\displaystyle\sum_{i=1}^2\int_{\mathbb{T}^{2}}\partial_iA_t(x_1,x_2)\psi^{\xi}\partial_ig(x_1,x_2) dx_1dx_2\\
&=\displaystyle\int_{\mathbb{T}^{2}}\Delta\psi^{\xi}g(x_1,x_2) dx_1dx_2,
\end{align*}
where we used \eqref{sca} with $\varphi=\psi^{\xi}(t,.)$. This is the weak formulation of:
$$\partial_t\psi^{\xi}=\Delta\psi^{\xi},\, \text{on}\,[0,\infty[\times \mathbb{T}^2.$$

\begin{remark}\label{simpdi}
The reason why we consider the weighted Helmholtz decomposition~\eqref{hel} with $\varphi=\psi^{\xi}(t,.)$ in the PABF dynamics \eqref{pabf} instead of the standard one~\eqref{helm0} is precisely to obtain this simple diffusion equation on the function $\psi^{\xi}$. This is will also be useful in the proof of Lemma~\ref{neg} below.
\end{remark}

\subsection{\textit{Proof of Corollary~\ref{fishco}}}
Let $\phi=\psi^{\xi}$ and $\phi_{\infty}=\psi^{\xi}_{\infty}=1$. It is known that $\forall t\geq0$ and $\forall (x_1,x_2)\in \mathbb{T}^2$, $\phi$ satisfies: 
\begin{equation}\label{diff2}
\partial_t\phi=\Delta_{x_1^2}\phi.
\end{equation}
Moreover (See Remark~\ref{rmdif}), it is assumed that
 and is such that $$\displaystyle\int_{\mathbb{T}^2}\phi(0,.)=1 \,\mbox{and}\,\phi(0,.)\geq0.$$
Let us show that $\forall\,t\geq 0,\forall k>0,\,\|\phi(t,.)-1\|_{H^k({\mathbb{T}^2})}\leq \|\phi(0,.)-1\|_{H^k({\mathbb{T}^2})}e^{-8\pi^2 t}.$

First, we prove that $\phi$ converges to $1$ in $L^2({\mathbb{T}^2})$,
\begin{align*}
\frac{1}{2}\frac{d}{dt}\displaystyle\int_{\mathbb{T}^2} | \phi-1|^2&=\displaystyle\int_{\mathbb{T}^2} \partial_t\phi(\phi-1)\\
&=\displaystyle\int_{\mathbb{T}^2}\Delta\phi(\phi-1)\\
&=-\displaystyle\int_{\mathbb{T}^2}\nabla\phi\nabla(\phi-1)\\
&=-\displaystyle\int_{\mathbb{T}^2}|\nabla\phi|^2\\
&\leq -4\pi^2\displaystyle\int_{\mathbb{T}^2}|\phi-1|^2,
\end{align*}
where we have used the Poincar\'e-Wirtinger inequality on the torus $\mathbb{T}^2$, applied to $\phi$: for any function $f\in H^1(\mathbb{T}^2)$,
$$\displaystyle \int_{\mathbb{T}^2}\left(f-\int_{\mathbb{T}^2}f\right)^2\leq \frac{1}{4\pi^2}\int_{\mathbb{T}^2}|\nabla f|^2.$$
We therefore obtain, $\|\phi(t,.)-1\|^2_{L^2(\mathbb{T}^2)}\leq\|\phi(0,.)-1\|^2_{L^2(\mathbb{T}^2)}e^{-8\pi^2t}$.

 Second, we prove that $\partial_i\phi$ converges to $0$ in $L^2({\mathbb{T}^2})$. For $i=1,2$, $\partial_i\phi$ satisfies~\eqref{diff2}: $\partial_t(\partial_i\phi)=\Delta_{x_1^2}(\partial_i\phi)$, with periodic boundary conditions. As above,
\begin{align*}
\frac{1}{2}\frac{d}{dt}\displaystyle\int_{\mathbb{T}^2} | \partial_i\phi|^2&=\displaystyle\int_{\mathbb{T}^2} \partial_t(\partial_i\phi)\partial_i\phi\\
&=\displaystyle\int_{\mathbb{T}^2}\Delta(\partial_i\phi)\partial_i\phi\\
&=-\displaystyle\int_{\mathbb{T}^2}|\nabla(\partial_i\phi)|^2.
\end{align*}
Using again Poincar\'e-Wirtinger inequality on $\partial_i\phi$,
\begin{align*}
\frac{1}{2}\frac{d}{dt}\displaystyle\int_{\mathbb{T}^2} | \partial_i\phi|^2&\leq-4\pi^2\displaystyle\int_{\mathbb{T}^2}\left(\partial_i\phi-\displaystyle\int_{\mathbb{T}^2}\partial_i\phi\right)^2\\
&=-4\pi^2\displaystyle\int_{\mathbb{T}^2}|\partial_i\phi|^2.
\end{align*}
Where we used $\int_{\mathbb{T}^2}\partial_i\phi=0$, since $\phi$ is periodic on $\mathbb{T}^2$. Therefore, it holds $\|\partial_i\phi(t,.)\|^2_{L^2(\mathbb{T}^2)}\leq\|\partial_i\phi(0,.)\|^2_{L^2(\mathbb{T}^2)}e^{-8\pi^2t}$.

Third, one can prove by induction that all higher derivatives of $\phi$ converge exponentially fast to $0$, with rate $8\pi^2$ and the following estimation is then proven: 
$$\forall\,t\geq 0,\forall k>0,\quad\|\phi(t,.)-1\|^2_{H^k({\mathbb{T}^2})}\leq \|\phi(0,.)-1\|_{H^k({\mathbb{T}^2})}^2e^{-8\pi^2 t}.$$
As $\displaystyle H^k(\mathbb{T}^2)\hookrightarrow L^{\infty}(\mathbb{T}^2),\,\forall k>1$, then $\exists c>0$,
$$\|\phi-1\|^2_{ L^{\infty}}\leq c\|\phi-1\|^2_{H^k}\leq ce^{-8\pi^2 t}.$$
Therefore, $\forall \varepsilon>0$, $\exists t_0>0$, $\forall x\in \mathbb{T}^2$, $\forall t>t_0$, $\phi(t,x)\geq 1-\varepsilon$. Finally, $\forall t>t_0$\\
$I(\psi^{\xi}|\psi^{\xi}_{\infty})=\displaystyle\int_{\mathbb{T}^2}\frac{|\nabla_{x_1^2}\phi|^2}{\phi}\leq\frac{1}{1-\varepsilon}\displaystyle\int_{\mathbb{T}^2}|\nabla_{x_1^2}\phi|^2\leq \frac{\|\nabla_{x_1^2}\phi(0,.)\|^2_{L^2(\mathbb{T}^2)}}{1-\varepsilon}e^{-8\pi^2t}.$

\subsection{\textit{Proof of Corollary \ref{mic}}}

We have that $\psi^{\xi}_{\infty}=1$ satisfies $LSI(r)$, for some $r>0$ (see Chapter~3, Section~3 in [\refcite{ane:00}]). Referring to Proposition~\ref{diff} and since $\psi^{\xi}$ is a probability density function, one gets: 
\begin{align*}
\frac{d}{dt}E_M&=\displaystyle\int_{\mathbb{T}^2}\partial_t\left(\psi^{\xi}\ln(\psi^{\xi})\right)\\
&=\displaystyle\int_{\mathbb{T}^2}\partial_t\psi^{\xi}\ln(\psi^{\xi})+\int_{\mathbb{T}^2}\partial_t\psi^{\xi}\\
&=\int_{\mathbb{T}^2}\Delta\psi^{\xi}\ln(\psi^{\xi})\\
&=-\int_{\mathbb{T}^2}|\nabla_{x_1^2}\ln(\psi^{\xi})|^2\psi^{\xi}\\
&=-I(\psi^{\xi}|\psi^{\xi}_{\infty})\\
&\leq-2r H(\psi^{\xi}|\psi^{\xi}_{\infty})\\
&=-2r E_M.
\end{align*}
Therefore, $E_M$ converges exponentially fast to zero. Referring to Corollary~\ref{fishco} and since $E_M$ converges to zero, we have that for any $t>t_0$,
\begin{align*}
\displaystyle -E_M(t)=\int_t^{\infty}\frac{d}{ds}E_M(s)ds&=-\int_t^{\infty}I(\psi^{\xi}|\psi^{\xi}_{\infty})ds\\
&\geq -I_0\int_t^{\infty}\mathrm{e}^{-8\pi^2s}ds\\
&=-\frac{I_0}{8\pi^2}\mathrm{e}^{-8\pi^2t},
\end{align*}
which yields the desired estimation.

\subsection{\textit{Proof of Theorem \ref{theo1}}}
To prove our main result, several intermediate lemmas are needed.
\begin{lemma}\label{dif}
$\forall t\geq 0$, $\forall (x_1,x_2)\in \mathbb{T}^2$ and for $i=1,2$, we have:
$$(F_t^i-\partial_iA)(x_1,x_2)=\displaystyle\left(\int_{\Sigma_{(x_1,x_2)}}\partial_i\ln(\psi/\psi_{\infty})\frac{\psi}{\psi^{\xi}}dx_3^n\right)(x_1,x_2)-\left(\partial_i\ln(\psi^{\xi}/\psi^{\xi}_{\infty})\right)(x_1,x_2).$$
\end{lemma}

\begin{proof}
\begin{align*}
&\displaystyle\int_{\Sigma_{(x_1,x_2)}}\partial_i\ln(\psi/\psi_{\infty})\frac{\psi}{\psi^{\xi}}dx_3^n-\partial_i\ln(\psi^{\xi}/\psi^{\xi}_{\infty})\\
&=\displaystyle\int_{\Sigma_{(x_1,x_2)}}\partial_i\ln(\psi)\frac{\psi}{\psi^{\xi}}dx_3^n-\int_{\Sigma_{(x_1,x_2)}}\partial_i\ln(\psi_{\infty})\frac{\psi}{\psi^{\xi}}dx_3^n-\partial_i\ln(\psi^{\xi})+\partial_i\ln(\psi^{\xi}_{\infty})\\
&=\displaystyle\frac{1}{\psi^{\xi}}\int_{\Sigma_{(x_1,x_2)}}\partial_i\psi dx_3^n+\int_{\Sigma_{(x_1,x_2)}}(\partial_i V-\nabla\xi_i\partial_iA\circ\xi)\frac{\psi}{\psi^{\xi}}dx_3^n-\partial_i\ln(\psi^{\xi})\\
&=\displaystyle\frac{\partial_i\psi^{\xi}}{\psi^{\xi}}+F_t^i-\partial_iA-\frac{\partial_i\psi^{\xi}}{\psi^{\xi}}\\
&=\displaystyle F_t^i-\partial_iA.
\end{align*}
\end{proof}

\begin{lemma}\label{difabs}
Suppose that [H1] and [H2] hold, then for all $t\geq0$, for all $(x_1,x_2)\in {\mathbb{T}^2}$ and for $i=1,2$, we have:
$$\displaystyle |F_t^i(x_1,x_2)-\partial_iA(x_1,x_2)|\leq \gamma\sqrt{\frac{2}{\rho}e_m(t,x_1,x_2)}.$$
\end{lemma}
\begin{proof}
For any coupling measure $\pi\displaystyle\in\prod(\mu_{t,x_1,x_2},\mu_{\infty,x_1,x_2})$ defined on $\Sigma_{(x_1,x_2)}\times\Sigma_{(x_1,x_2)}$, it holds:
\begin{align*}
\displaystyle |F_t^i-\partial_iA|&=\left|\displaystyle\int_{\Sigma_{(x_1,x_2)}\times\Sigma_{(x_1,x_2)}}(\partial_iV(x)-\partial_iV(x'))\pi(dx,dx')\right|\\
&=\|\nabla_{x_3^n}\partial_iV\|_{L^{\infty}}\sqrt{\displaystyle\int_{\Sigma_{(x_1,x_2)}\times\Sigma_{(x_1,x_2)}}d_{\Sigma_{(x_1,x_2)}}(x,x')^2\pi(dx,dx')}.
\end{align*}
Taking now the infimum over all $\pi\displaystyle\in\prod(\mu(t,.|(x_1,x_2)),\mu(\infty,.|(x_1,x_2)))$ and using Lemma \ref{tala}, we obtain
\begin{align*}
\displaystyle |F_t^i-\partial_iA|&\leq\displaystyle \gamma W(\mu(t,.|(x_1,x_2)),\mu^{\xi}(\infty,.|(x_1,x_2)))\\
&\leq\displaystyle \gamma\sqrt{\frac{2}{\rho}H(\mu^{\xi}(t,.|(x_1,x_2)),\mu^{\xi}(\infty,.|(x_1,x_2)))}\\
&=\displaystyle \gamma\sqrt{\frac{2}{\rho}e_m(t,(x_1,x_2))}.\\
\end{align*}
\end{proof}

\begin{lemma}\label{em}
Suppose that [H2] holds, then for all $t\geq0$,
$$E_m(t)\leq\displaystyle\frac{1}{2\rho}\int_{\mathbb{T}^n}|\nabla_{x_3^n}\ln(\psi(t,.)/\psi_{\infty})|^2\psi.$$
\end{lemma}

\begin{proof}
Using [H2],
\begin{align*}
E_m&=\displaystyle\int_{\mathbb{T}^2}e_m\psi^{\xi}dx_1dx_2\\
&=\displaystyle\int_{\mathbb{T}^2}H(\mu(t,.|(x_1,x_2))|\mu(\infty,.|(x_1,x_2)))\psi^{\xi}dx_1dx_2\\
&\leq \displaystyle\int_{\mathbb{T}^2}\frac{1}{2\rho}\int_{\Sigma_{x_1^2}}|\nabla_{x_3^n}\ln(\psi(t,.)/\psi_{\infty})|^2dx_3^n\frac{\psi(t,.)}{\psi^{\xi}(t,x_1,x_2)}dx_1dx_2\\
&=\displaystyle\frac{1}{2\rho}\int_{\mathbb{T}^n}|\nabla_{x_3^n}\ln(\psi(t,.)/\psi_{\infty})|^2\psi dx_1^n.\\
\end{align*}
\end{proof}

\begin{lemma}\label{neg}
It holds for all $t\geq0$,
$$\displaystyle\int_{\mathbb{T}^n}(\partial_1A_t-F_t^1)[\partial_1\ln(\psi/\psi_{\infty})]\psi+\int_{\mathbb{T}^n}(\partial_2A_t-F_t^2)[\partial_2\ln(\psi/\psi_{\infty})]\psi\leq 0.$$
\end{lemma}

\begin{proof}
Using Fubini,
\begin{align*}
&\displaystyle\int_{\mathbb{T}^n}(\partial_1A_t-F_t^1)[\partial_1\ln(\psi/\psi_{\infty})]\psi+\int_{\mathbb{T}^n}(\partial_2A_t-F_t^2)[\partial_2\ln(\psi/\psi_{\infty})]\psi\\
&=\displaystyle\int_{\mathbb{T}^2}(\partial_1A_t-F_t^1)\int_{\Sigma_{(x_1,x_2)}}[\partial_1\ln(\psi/\psi_{\infty})]\psi+\int_{\mathbb{T}^2}(\partial_2A_t-F_t^2)\int_{\Sigma_{(x_1,x_2)}}[\partial_2\ln(\psi/\psi_{\infty})]\psi.
\end{align*}
For the first term, we have
\begin{align*}
\displaystyle\int_{\Sigma_{(x_1,x_2)}}[\partial_1\ln(\psi/\psi_{\infty})]\psi&=\displaystyle\int_{\Sigma_{(x_1,x_2)}}(\partial_1\ln\psi)\psi-\int_{\Sigma_{(x_1,x_2)}}(\partial_1\ln\psi_{\infty})\psi\\
&=\partial_1\psi^{\xi}+\int_{\Sigma_{(x_1,x_2)}}\partial_1(V-A)\psi\\
&=(\partial_1\ln\psi^{\xi})\psi^{\xi}+F_t^1\psi^{\xi}-\partial_1A\psi^{\xi}.
\end{align*}
Similarly, we have
$$\displaystyle\int_{\Sigma_{(x_1,x_2)}}[\partial_2\ln(\psi/\psi_{\infty})]\psi=(\partial_2\ln\psi^{\xi})\psi^{\xi}+F_t^2\psi^{\xi}-\partial_2A\psi^{\xi}.$$
Therefore, one gets
\begin{align*}
&\displaystyle\int_{\mathbb{T}^n}(\partial_1A_t-F_t^1)[\partial_1\ln(\psi/\psi_{\infty})]\psi+\int_{\mathbb{T}^n}(\partial_2A_t-F_t^2)[\partial_2\ln(\psi/\psi_{\infty})]\psi\\
&=\displaystyle\int_{\mathbb{T}^2}(\partial_1A_t-F_t^1)(\partial_1\ln\psi^{\xi})\psi^{\xi}+\int_{\mathbb{T}^2}(\partial_2A_t-F_t^2)(\partial_2\ln\psi^{\xi})\psi^{\xi}\\
&\quad -\displaystyle\int_{\mathbb{T}^2}(\partial_1A_t-F_t^1)^2\psi^{\xi}-\int_{\mathbb{T}^2}(\partial_2A_t-F_t^2)^2\psi^{\xi}
\end{align*}
which concludes the assertion since the first line is equal to zero (by \eqref{sca} with $\varphi=\psi^{\xi}(t,.)$) and the second line is non positive. Again the weighted Helmholtz decomposition helps in simplifying terms.
\end{proof}

\noindent\textbf{Proof of Theorem~\ref{theo1}:}\\

Now we will prove the exponentially convergence of $E_m(t)$. Recall \eqref{reform}:
$$\partial_t\psi=\displaystyle{\rm div}(\nabla V\psi+\nabla\psi)-\partial_1((\partial_1A_t)\psi)-\partial_2((\partial_2A_t)\psi),$$
 which is equivalent to
$$\partial_t\psi=\displaystyle{\rm div}\left(\psi_{\infty}\nabla\left(\frac{\psi}{\psi_{\infty}}\right)\right)+\partial_1[(\partial_1A-\partial_1A_t)\psi]+\partial_2[(\partial_2A-\partial_2A_t)\psi].$$
Using \eqref{tent}, \eqref{maent} and \eqref{diff2}, one obtains
\begin{align*}
\displaystyle\frac{dE}{dt}&=-\displaystyle\int_{\mathbb{T}^n}|\nabla\ln(\psi/\psi_{\infty})|^2\psi +\int_{\mathbb{T}^n}\!(\partial_1A_t-\partial_1A)[\partial_1\ln(\psi/\psi_{\infty})]\psi\\
&\quad+ \displaystyle\int_{\mathbb{T}^n}\!(\partial_2A_t-\partial_2A)[\partial_2\ln(\psi/\psi_{\infty})]\psi,\\
\displaystyle\frac{dE_M}{dt}&=-\displaystyle\int_{\mathbb{T}^2}|\nabla\ln(\psi^{\xi})|^2\psi^{\xi}.
\end{align*}
\noindent Using then Lemma \ref{entsum} and Lemma \ref{dif}, one gets
\begin{align*}
\displaystyle\frac{dE_m}{dt}&=\displaystyle\frac{dE}{dt}-\displaystyle\frac{dE_{M}}{dt}\\
&=-\displaystyle\int_{\mathbb{T}^n}|\nabla\ln(\psi/\psi_{\infty})|^2\psi+\int_{\mathbb{T}^n}\!(\partial_1A_t-\partial_1A)\partial_1\ln(\psi/\psi_{\infty})\psi\\
&\quad+ \displaystyle\int_{\mathbb{T}^n}\!(\partial_2A_t-\partial_2A)\partial_2\ln(\psi/\psi_{\infty})\psi +\displaystyle\int_{\mathbb{T}^2}|\partial_1\ln\psi^{\xi}|^2\psi^{\xi}+\int_{\mathbb{T}^2}|\partial_2\ln\psi^{\xi}|^2\psi^{\xi}\\
&=-\displaystyle\int_{\mathbb{T}^n}|\nabla\ln(\psi/\psi_{\infty})|^2\psi\\
&\quad \displaystyle+\int_{\mathbb{T}^n}\!(\partial_1A_t-F_t^1)[\partial_1\ln(\psi/\psi_{\infty})]\psi+ \int_{\mathbb{T}^n}\!(F_t^1-\partial_1A)[\partial_1\ln(\psi/\psi_{\infty})]\psi\\
&\quad \displaystyle+\int_{\mathbb{T}^n}\!(\partial_2A_t-F_t^2)[\partial_2\ln(\psi/\psi_{\infty})]\psi+ \int_{\mathbb{T}^n}\!(F_t^2-\partial_2A)[\partial_2\ln(\psi/\psi_{\infty})]\psi\\
&\quad +\displaystyle\int_{\mathbb{T}^2}|\partial_1\ln\psi^{\xi}|^2\psi^{\xi}+\int_{\mathbb{T}^2}|\partial_2\ln\psi^{\xi}|^2\psi^{\xi}.
\end{align*}
\noindent Lemma \ref{neg} then yields
\begin{align*}
\displaystyle\frac{dE_m}{dt}&\leq-\displaystyle\int_{\mathbb{T}^n}|\nabla\ln(\psi/\psi_{\infty})|^2\psi\\
&\quad \displaystyle+\int_{\mathbb{T}^n}(F_t^1-\partial_1A)\partial_1\ln(\psi/\psi_{\infty})\psi+\int_{\mathbb{T}^n}\!(F_t^2-\partial_2A)\partial_2\ln(\psi/\psi_{\infty})\psi \\
&\quad +\displaystyle\int_{\mathbb{T}^2}|\partial_1\ln\psi^{\xi}|^2\psi^{\xi}+\int_{\mathbb{T}^2}|\partial_2\ln\psi^{\xi}|^2\psi^{\xi}.
\end{align*}
Using lemma \ref{dif} and Fubini, one then obtains
\begin{align*}
\displaystyle\frac{dE_m}{dt}&\leq-\displaystyle\int_{\mathbb{T}^n}|\nabla\ln(\psi/\psi_{\infty})|^2\psi\\
&\quad +\displaystyle\int_{\mathbb{T}^2}\left[\int_{\Sigma_{(x_1,x_2)}}\!\!\!\!\!\!\!\!\!\!\partial_1\ln(\psi/\psi_{\infty})\frac{\psi}{\psi^{\xi}}\right]\int_{\Sigma_{(x_1,x_2)}}\!\!\!\!\!\!\!\!\!\!\partial_1\ln(\psi/\psi_{\infty})\psi
-\displaystyle\int_{\mathbb{T}^n}\!\!\!\!\partial_1\ln(\psi^{\xi})\partial_1\ln(\psi/\psi_{\infty})\psi\\
&\quad +\displaystyle\int_{\mathbb{T}^2}\left[\int_{\Sigma_{(x_1,x_2)}}\!\!\!\!\!\!\!\!\!\!\partial_2\ln(\psi/\psi_{\infty})\frac{\psi}{\psi^{\xi}}\right]\int_{\Sigma_{(x_1,x_2)}}\!\!\!\!\!\!\!\!\!\!\partial_2\ln(\psi/\psi_{\infty})\psi
-\displaystyle\int_{\mathbb{T}^n}\!\!\!\!\partial_2\ln(\psi^{\xi})\partial_2\ln(\psi/\psi_{\infty})\psi\\
&\quad +\displaystyle\int_{\mathbb{T}^2}|\partial_1\ln\psi^{\xi}|^2\psi^{\xi}+\int_{\mathbb{T}^2}|\partial_2\ln\psi^{\xi}|^2\psi^{\xi}\\
&\leq -\displaystyle\int_{\mathbb{T}^n}|\nabla\ln(\psi/\psi_{\infty})|^2\psi\\
&\quad +\displaystyle\int_{\mathbb{T}^2}\left[\int_{\Sigma_{(x_1,x_2)}}\partial_1\ln(\psi/\psi_{\infty})\psi\right]^2\frac{1}{\psi^{\xi}}-\displaystyle\int_{\mathbb{T}^n}\partial_1\ln(\psi^{\xi})\partial_1\ln(\psi/\psi_{\infty})\psi\\
&\quad +\displaystyle\int_{\mathbb{T}^2}\left[\int_{\Sigma_{(x_1,x_2)}}\partial_2\ln(\psi/\psi_{\infty})\psi\right]^2\frac{1}{\psi^{\xi}}-\displaystyle\int_{\mathbb{T}^n}\partial_2\ln(\psi^{\xi})\partial_2\ln(\psi/\psi_{\infty})\psi\\
&\quad +\displaystyle\int_{\mathbb{T}^2}|\partial_1\ln\psi^{\xi}|^2\psi^{\xi}+\int_{\mathbb{T}^2}|\partial_2\ln\psi^{\xi}|^2\psi^{\xi}.
\end{align*}
Applying Cauchy-Schwarz on the first terms of the second and third lines, we obtain
\begin{align*}
\displaystyle\frac{dE_m}{dt}&\leq -\displaystyle\int_{\mathbb{T}^n}|\nabla_{x_3^n}\ln(\psi/\psi_{\infty})|^2\psi\\
&\quad -\displaystyle\int_{\mathbb{T}^n}\partial_1\ln(\psi^{\xi})\partial_1\ln(\psi/\psi_{\infty})\psi-\int_{\mathbb{T}^n}\partial_2\ln(\psi^{\xi})\partial_2\ln(\psi/\psi_{\infty})\psi\\
&\quad +\displaystyle\int_{\mathbb{T}^2}|\partial_1\ln\psi^{\xi}|^2\psi^{\xi}+\int_{\mathbb{T}^2}|\partial_2\ln\psi^{\xi}|^2\psi^{\xi}\\
&\leq-\displaystyle\int_{\mathbb{T}^n}|\nabla_{x_3^n}\ln(\psi/\psi_{\infty})|^2\psi-\displaystyle\int_{\mathbb{T}^2}\partial_1\ln(\psi^{\xi})\left[\int_{\Sigma_{(x_1,x_2)}}\!\!\!\!\!\!\!\!\!\!\partial_1\ln(\psi/\psi_{\infty})\frac{\psi}{\psi^{\xi}}-\partial_1\ln\psi^{\xi}\right]\psi^{\xi}\\
&\quad -\displaystyle\int_{\mathbb{T}^2}\partial_2\ln(\psi^{\xi})\left[\int_{\Sigma_{(x_1,x_2)}}\partial_2\ln(\psi/\psi_{\infty})\frac{\psi}{\psi^{\xi}}-\partial_2\ln\psi^{\xi}\right]\psi^{\xi}.
\end{align*}
Applying Lemma~\ref{em}, Lemma~\ref{dif}, Cauchy-Schwarz, Lemma~\ref{difabs} and Corollary~\ref{fishco},
\begin{align*}
\displaystyle\frac{dE_m}{dt}&\leq -2\rho E_m+\displaystyle\sqrt{\int_{\mathbb{T}^2}|\partial_1\ln(\psi^{\xi})|^2\psi^{\xi}}\sqrt{\int_{\mathbb{T}^2}\frac{2}{\rho}e_m(t,(x_1,x_2))\psi^{\xi}}\\
&\quad +\displaystyle\sqrt{\int_{\mathbb{T}^2}|\partial_2\ln(\psi^{\xi})|^2\psi^{\xi}}\sqrt{\int_{\mathbb{T}^2}\frac{2}{\rho}e_m(t,(x_1,x_2))\psi^{\xi}}\\
&\leq \displaystyle-2\rho E_m+2\gamma\sqrt{\frac{2}{\rho}E_m}\sqrt{\int_{\mathbb{T}^2}|\nabla_{x_1^2}\ln(\psi^{\xi})|^2\psi^{\xi}}\\
&\leq\displaystyle-2\rho E_m+2\gamma\sqrt{\frac{2}{\rho}E_m}\sqrt{I(\psi^{\xi}/\psi^{\xi}_{\infty})}\\
&\leq\displaystyle-2\rho E_m+2\gamma\sqrt{\frac{2}{\rho}E_m}\sqrt{I_0}e^{-4\pi^2t}.
\end{align*}
Finally we obtain
\begin{align*}
\frac{d}{dt}\sqrt{E_m(t)}&\leq\displaystyle-\rho\sqrt{E_m(t)}+\gamma\sqrt{\frac{I_0}{2\rho}}e^{-4\pi^2t}.
\end{align*}
First, if $\rho\neq 4\pi^2$, using Gronwall inequality, one obtains
\begin{align*}
\sqrt{E_m(t)}&\leq\sqrt{E_m(0)}\,\displaystyle e^{-\rho t}+\gamma\sqrt{\frac{I_0}{2\rho}}\displaystyle\int_0^te^{\rho(-t+s)}e^{-4\pi^2s}ds\\
&\leq \displaystyle\sqrt{E_m(0)}\,e^{-\rho t}+\gamma\sqrt{\frac{I_0}{2\rho}}\frac{e^{-\rho t}}{\rho-4\pi^2}\left(e^{(\rho-4\pi^2)t}-1\right)\\
&\leq \displaystyle\sqrt{E_m(0)}\,e^{-\rho t}+\gamma\sqrt{\frac{I_0}{2\rho}}\frac{e^{-\rho t}}{|\rho-4\pi^2|}e^{(\rho-4\pi^2)t}.
\end{align*}
    \begin{align*}
\sqrt{E_m(t)}&\leq\displaystyle\sqrt{E_m(0)}\,e^{-\rho t}+\gamma\sqrt{\frac{I_0}{2\rho}}\displaystyle\int_0^te^{-\rho t}ds\\
&\leq \displaystyle\left(\sqrt{E_m(0)}+\gamma\sqrt{\frac{I_0}{2\rho}}t\right)e^{-\rho t},
\end{align*}
which leads the desired estimation \eqref{Em}.
  
Using this above convergence, Corollary \ref{mic} and Lemma \ref{entsum}, it is then easy to see that $E$ converges exponentially fast to zero. Using \eqref{kul}, one obtains the convergence of $\psi$ to $\psi_{\infty}$ since:
$$\|\psi-\psi_{\infty}\|_{L^1({\mathbb{T}^n})}\leq\sqrt{2H(\psi|\psi_{\infty})}=\sqrt{2E}.$$
The second point of the theorem is checked. Finally, we are now in position to prove the last point of Theorem \ref{theo1}. Using \eqref{minw} and Lemma \ref{difabs},
\begin{align*}
\|\nabla A_t-\nabla A\|_{L^2_{\psi^{\xi}}(\mathbb{T}^2)}^2&\leq 2\|\nabla A_t-F_t\|_{L^2_{\psi^{\xi}}(\mathbb{T}^2)}^2+2\|F_t-\nabla A\|_{L^2_{\psi^{\xi}}(\mathbb{T}^2)}^2\\
&\leq 4\|F_t-\nabla A\|_{L^2_{\psi^{\xi}}(\mathbb{T}^2)}^2\\
&\leq 8\frac{\gamma^2}{\rho}E_m.
\end{align*}

\section*{Acknowledgment}
The authors thank Gabriel Stoltz for very helpful discussions on the numerical experiments. Houssam Alrachid would like to thank the Ecole des Ponts ParisTech and CNRS Libanais for supporting his PHD thesis. The work of Tony Leli\`evre is supported by the European Research Council under the European Union's Seventh Framework Programme (FP/2007-2013) / ERC Grant Agreement number 614492.

\end{document}